\documentclass[12pt,reqno]{amsart}

\usepackage{amscd}

\usepackage{amsthm}
\usepackage{amstext}
\usepackage{amssymb}
\usepackage{amsmath,amscd}
\usepackage{enumerate}
\usepackage{mathrsfs}
\usepackage{amsrefs}
\usepackage[utf8]{inputenc}
\usepackage{geometry}
\usepackage{xcolor}
\usepackage{adjustbox}
\usepackage[all]{xy}
\usepackage{stmaryrd}
\usepackage{tikz}
\usepackage{booktabs}
\usepackage{longtable}
\usepackage{hyperref}
\usepackage{caption}
\numberwithin{equation}{section}

\newcommand{\mat}[2][ccccccccccccccccccccccccc]{
  \left(
    \begin{array}{#1}
      #2\\
    \end{array}
  \right)
}

\newcommand{\Fr}{\textnormal{Fr}}

\newcommand{\TFr}{\textnormal{TFr}}

\newcommand{\Tr}{\textnormal{Tr}}

\newtheorem{theorem}{Theorem}[section]
\newtheorem{proposition}[theorem]{Proposition}
\newtheorem{lemma}[theorem]{Lemma}

\newtheorem{conjecture}[theorem]{Conjecture}

\theoremstyle{definition}
\newtheorem{definition}[theorem]{Definition}
\newtheorem{remark}[theorem]{Remark}
\newtheorem{example}[theorem]{Example}

\newcommand{\cB}{\mathcal B}

\newcommand{\cR}{\mathcal R}

\newcommand{\C}{\mathbb C}
\newcommand{\F}{\mathbb F}

\newcommand{\N}{\mathbb N}
\newcommand{\Q}{\mathbb Q}
\newcommand{\R}{\mathbb R}
\newcommand{\Z}{\mathbb Z}

\title{Invertible Calabi--Yau Orbifolds over Finite Fields II}

\author{Marco Aldi}
\address{Department of Mathematics and Applied Mathematics\\
Virginia Commonwealth University\\
Richmond, VA 23284, USA}

\author{Andrija Peruni\v{c}i\'{c}}
\address{}

\begin{document}

\begin{abstract}
We state a conjecture about the zeta function of crepant resolutions of Berglund--H\"ubsch orbifold hypersurfaces over a finite field. In addition to numerical evidence, we show that our conjectural zeta function satisfies the Weil conjectures and we elucidate its connection with Monsky--Washnitzer cohomology.
\end{abstract}

\maketitle

\section{Introduction}

Building upon the work of Greene and Plesser \cite{GP}, Berglund and H\"ubsch \cite{BH} formulated a systematic procedure for constructing mirror pairs of Calabi--Yau orbifolds. As part of an innovative vertex algebra approach to Berglund--H\"ubsch mirror symmetry, Borisov \cite{B1} introduced a complex whose cohomology calculates the Chen-Ruan cohomology of Berglund--H\"ubsch orbifolds. In \cite{AP} we introduced a deformation of Borisov's complex which admits a natural Frobenius endomorphism when defined $p$-adically. In \cite{AP2} we explicitly computed the eigenvalues of the (inverse) Frobenius endomorphism acting on the cohomology of the $p$-adic deformed Borisov complex. Based on these calculations,  we conjectured \cite{AP2} that the supertrace of the Frobenius endomorphism acting on the cohomology of the deformed Borisov complex computes the number of $\mathbb F_p$-points of any crepant resolution of Berglund--H\"ubsch orbifolds of Calabi--Yau type. So far, this conjecture has been proved for all Berglund--H\"ubsch elliptic curves, for diagonal matrices with pairwise coprime weights, and for all $14$ weighted diagonal K3 surfaces classified in~\cite{Goto}.

In this paper, we provide additional evidence for the conjecture stated in \cite{AP2}. Our starting point is the observation that this conjecture can be naturally upgraded to a statement (Conjecture~\ref{conj:supertrace-counts-points}) about the counting of  $\mathbb F_{p^\nu}$-points of crepant resolutions of Berglund--H\"ubsch orbifolds. Motivated by this, we introduce a zeta function $\zeta_p(A,G,p)$ for the action of the Frobenius endomorphism on the cohomology of the deformed Borisov complex of a Berglund--H\"ubsch matrix $A$ over $\mathbb F_p$ with symmetry group $G$. Since Conjecture~\ref{conj:supertrace-counts-points} implies that in the Calabi--Yau case $\zeta_p(A,G,t)$ is equal to the local zeta function of a smooth projective variety, we expect $\zeta_p(A,G,t)$ to formally satisfy the Weil conjectures. As a non-trivial check for our conjecture, we prove that this is indeed the case. We also offer numerical evidence for our conjecture by comparing $\zeta_p(A,G,t)$ with known calculations of the local zeta function of several K3 surfaces \cites{Goto, GotoThesis, Whitcher21} and of the Fermat quintic threefold \cite{CdlORV03}. Furthermore, we calculate the supertrace of our Frobenius operator in the case of Greene--Plesser mirror quotients of the quintic with non-trivial symmetry groups $G\supsetneq\langle J\rangle$.

Finally, we describe an alternative proof strategy for our conjecture for smooth weighted homogeneous Berglund--H\"ubsch hypersurfaces with $G=\langle J\rangle$ by comparing the orbifold trace formula directly against the Monsky--Washnitzer point-counting formula~\cite{M}. Doing this requires carefully carrying out some delicate cancellations. We do so in detail for Fermat-type (diagonal) matrices and illustrate in an example additional subtleties that can arise with non-diagonal potentials.

The accompanying SageMath code---implementing the orbifold trace formula, zeta function assembly, and computational verifications---is available at \cite{P}.

\section{Preliminaries}\label{sec:preliminaries}

In this section we review the basic framework developed in \cites{AP, AP2}. Throughout the paper, vectors are thought of as row vectors unless otherwise specified and $e_i$ represents the $i$-th vector in the standard basis of $\mathbb Z^n$.

\begin{definition}\label{def:berglund-hubsch}
Let $p$ be a prime number. An $n\times n$ matrix with non-negative integer entries $A$ is {\it Berglund--H\"ubsch} over $\mathbb F_p$ if
\begin{enumerate}
\item $A$ is invertible;
\item the polynomial $W_A(x)=\sum_{i=1}^n x^{e_i A}$ is quasi-homogeneous and such that, in $\mathbb F_p[x]=\mathbb F_p[x_1,\ldots,x_n]$, $(\partial_{x_1}W_A,\ldots \partial_{x_n}W_A)$ is a regular sequence;
\item $\det(A)$ divides $p-1$.
\end{enumerate}

\end{definition}

 Given a Berglund--H\"ubsch matrix $A$ over $\mathbb F_p$, its transpose $A^{T}$ is also Berglund--H\"ubsch over $\mathbb F_p$ \cite{AP}. We denote by $\widetilde \cB_A$ be the subspace of $p$-adic overconvergent series in the even variables $\{x_i,y_i\}_{i=1}^n$ as well as in the odd variables $\{\theta_i\}_{i=1}^n$, spanned by those monomials $x^\gamma y^\lambda \theta^I$ such that $\lambda A^{-T}\ge 0$. We define $\mathcal B_A$ to be the quotient of $\widetilde \cB_A$ by the ideal generated by all monomials $x^\gamma y^\lambda \theta^I$ with the property that $\gamma A^{-1}\lambda^T>0$. There is a natural grading on $\cB_A$ such that the degree of the general monomial $x^\gamma y^\lambda \theta^I$ is equal to $q(\lambda,I)+q^\vee(\lambda,I)$, where $q^\vee(\lambda,I)$ is the number of indices $i\in\{1,\ldots,n\}$ such that  $(\lambda A^{-T})_i>(\lambda A^{-T})_iI_i$ and $q(\lambda,I)=I_1+\cdots+I_n+q^\vee(\lambda,I)$. For a fixed $\pi\in \C_p$ such that $\pi^{p-1}=-p$, we define the operator $D_A:\mathcal B_A\to \mathcal B_A$ such that
\begin{align*}
D_A(x^\gamma y^\lambda \theta^I)&=\sum_{i=1}^n\gamma_i x^\gamma y^\lambda \theta^{I+e_i} + \pi\sum_{i,j=1}^n A_{ji}\, x^{\gamma + e_j A}y^\lambda \theta^{I+e_i}\\
&+\sum_{i=1}^n\pi^{-1}(\lambda A^{-T})_i\, x^\gamma y^\lambda \partial_{\theta_i} \theta^I + \sum_{i=1}^nx^\gamma y^{\lambda + e_iA^T} \partial_{\theta_i} \theta^I
\end{align*}
for all monomials $x^\gamma y^\lambda \theta^I$ in $\mathcal B_A$.

\begin{proposition}[\cite{AP}]\label{prop:2.8}
Let $A$ be a Berglund--H\"ubsch matrix over $\mathbb F_p$. Then
\begin{enumerate}[1)]
    \item $(\mathcal B_A,D_A)$ is a chain complex.
    \item $(\cB_A,D_A)$ decomposes into the direct sum of sub-complexes $(\cB_A(\gamma,\lambda),D_A)$ labeled by classes $[(\gamma,\lambda)]\in \Z^{2n}/(\Z^{2n}(A\oplus A^T))\cong G_{A^T}\times G_A$.
\item For each $(\gamma,\lambda)\in G_{A^T}\times G_A$, the cohomology of $(\cB_A(\gamma,\lambda),D_A)$ is at most one dimensional.
\end{enumerate}
\end{proposition}

Consider a Berglund--H\"ubsch matrix $A$ over $\mathbb F_p$. Given a subgroup $G< G_A$, its {\it transpose} \cite{Kra} is the subgroup $G^T <G_{A^T}$ such that $\gamma\in G_{A^T}$ belongs to $G^T$ if and only if $\gamma A^{-1} \lambda^T \in \mathbb Z$ for all $\lambda\in G$.

If $\lambda\in G_A$, then the direct sum $\mathcal B_A^\lambda$ of all $\mathcal B_A(\gamma,\lambda)$ over all $\gamma\in G_{A^T}$ is a subcomplex of $(\mathcal B_A,D_A)$. Furthermore \cite{AP}, let $A^\lambda$ be the matrix such that $W_{A^\lambda}(x)$ is obtained from $W_A(x)$ by setting $x_i=0$ whenever $(\lambda A^{-T})_i\in \mathbb Q\setminus \mathbb Z$. Then $A^\lambda$ is Berglund--H\"ubsch over $\mathbb F_q$ and $H(\mathcal B_{A^\lambda}^0,D_{A^\lambda})\cong H(\mathcal B_A^\lambda,D_A)$.

Let $A$ be a Berglund--H\"ubsch matrix over $\mathbb F_p$ and let $G<G_A$ be a subgroup. We define $\mathcal B_A^G$ to be the direct sum of all $\mathcal B_A(\gamma,\lambda)$ over all $(\gamma,\lambda)\in G^T\times G$.

Consider an $n\times n$ Berglund--H\"ubsch matrix $A$ over $\mathbb F_q$. Assume the representative in the equivalence class of $\lambda\in G_A$ modulo $\mathbb Z^nA^T$ is chosen so that $0\le (\lambda A^{-T})_i<1$ for every $i\in \{1,\ldots,n\}$. The {\it age} of $\lambda$ is ${\rm age}(\lambda)=\sum_{i=1}^n(\lambda A^{-T})_i$. Similarly, the {\it dual age} of $\gamma\in G_{A^T}$ is defined to be ${\rm age}^\vee(\gamma)=\sum_{i=1}^n(\gamma A^{-1})_i$. We also define $\dim(\lambda)$ (resp.\ $\dim(\gamma)$) to be the number of indices $i$ for which $(\lambda A^{-T})_i\in \mathbb Z$ (resp.\ $(\gamma A^{-1})_i\in \mathbb Z$).

In what follows we assume that both $G$ and $G^T$ contain $J=e_1+\ldots+e_n$ so that the Calabi--Yau condition $JA^{-T}J^T\in \mathbb Z$ is satisfied and the age of any $\lambda\in G$ and $\gamma\in G^T$ is an integer.

\begin{definition} We define the {\it Frobenius endomorphism} $\Fr_A:\cB_A\to \cB_A$ by the formula
\begin{equation}
\Fr_A(x^\gamma y^\lambda \theta^I)=p^{-q(\lambda,I)} \sum c_k \pi^{|k|} x^{(\gamma+kA)/p}y^\lambda \theta^I\,,
\end{equation}
where the sum ranges over all $k=(k_1,\ldots,k_n)\in \N^n$ such that $\gamma+kA\in p\N^n$. The coefficients $c_k$ are such that $c_0=1$ and $c_k=0$ if $k_j<0$ for some $j$, while
\[
k_jc_k=\pi c_{k-e_j}+\pi^p c_{k-pe_j}
\]
otherwise. Given a non-negative integer $\nu$, we denote by $\Fr_A^\nu$ the $\nu$-fold iterated composition $ (\Fr_A)^\nu$ acting on $\cB_A$.
\end{definition}

It follows from \cite{AP2} that each $(\Fr_A)^\nu$ commutes with $D_A$ and thus descends to an operator $H(\Fr_A^\nu)$ acting on the cohomology of $(\cB_A,D_A)$. Furthermore, by Proposition 3.4 of \cite{AP2}, $H(\Fr_A)$ acts diagonally on the standard monomial basis for the cohomology of $(\cB_A,D_A)$. In particular, direct inspection shows that, whenever $(\gamma,\lambda)$ labels a cohomology generator, the eigenvalue of $H(\Fr_A)$ on $H(\cB_A(\gamma,\lambda),D_A)$ is
\begin{equation}\label{eq:epsilon}
\epsilon(\gamma,\lambda)=(-p)^{{\rm age}^\vee(\gamma)}p^{-n}\Gamma_p(\gamma A^{-1})\,.
\end{equation}

\begin{definition}\label{def:twisted-frob}
The {\it twisted Frobenius} $\TFr_A$ is the operator on $H(\cB_A^G,D_A)$ defined by $\TFr_A=p^{{\rm age}+n-1}H(\Fr_A)$, where $p^{\rm age}$ denotes the operator diagonalized in the monomial basis with eigenvalue $p^{{\rm age}(\lambda)}$ on $H(\cB_A(\gamma,\lambda),D_A)$.
\end{definition}

\begin{remark}
The twist by $p^{{\rm age}+n-1}$ has a natural geometric origin. Let $X_A$ be the hypersurface obtained by setting $W_A=0$ in weighted projective space with weights $w_i=m(JA^{-T})_i$, where $m$ is the smallest integer such that each $w_i$ is an integer. Let $Y_{A,G}$ be the orbifold $[X_A/(G/\langle J\rangle)]$. The standard shift by the dimension of $Y_{A,G}$ needed for the ``arithmetic'' Frobenius to count points in the sense of \cite{Behrend93} is $p^{n-2}$. The additional factor $p^{\rm age+1}$ accounts for the contribution of twisted sectors as studied in \cite{Rose07}, where the action of the geometric Frobenius on orbifolds is modified by a power of $p$.
\end{remark}

\begin{definition}\label{def:supertrace}
Let $A$ be an $n\times n$ Berglund--H\"ubsch matrix over $\mathbb F_p$ and let $G$ be a subgroup of $G_A$ such that $J\in G\cap G^T$. For any positive integer $\nu$, we denote the supertrace of $\TFr_A^\nu=(p^{{\rm age}+n-1}H(\Fr_A))^\nu$ acting on $H(\cB_A^G,D_A)$ by
\begin{equation}\label{eq:supertrace-nu}
{\rm ST}_{p^\nu}(A,G)=\sum_{\gamma\in G^T,\lambda\in G} (-1)^{\dim(\lambda)+\nu\,{\rm age}^\vee(\gamma)} p^{\nu({\rm age}(\lambda)+{\rm age}^\vee(\gamma)-1)} \Gamma_p(\gamma A^{-1})^\nu\,\delta(\gamma,\lambda)\,,
\end{equation}
where $\delta(\gamma,\lambda)\in \{0,1\}$ is the dimension of the total cohomology of the subcomplex $(\cB_A(\gamma,\lambda),D_A)$.
\end{definition}

\section{Conjectures}

In this section we discuss our main conjecture regarding point counting, its equivalent formulation in the language of zeta functions, and the necessity of some of the assumptions made.

\begin{conjecture}\label{conj:supertrace-counts-points}
Let $A$ be an $n\times n$ Berglund--H\"ubsch matrix over $\F_p$ with $n\ge 3$, and let $G$ be a subgroup of $G_A$ such that $J\in G\cap G^T$. If $Y_{A,G}$ is the orbifold $[X_A/(G/\langle J\rangle)]$ and $Z_{A,G}$ is any crepant resolution of $Y_{A,G}$, then for all $\nu\ge 1$,
\[
\#Z_{A,G}(\F_{p^\nu}) = {\rm ST}_{p^\nu}(A,G).
\]
At $\nu=1$, this is \cite[Conjecture~4.1]{AP2}.
\end{conjecture}

\begin{remark}\label{rem:eigenvalue-remark}
By \eqref{eq:epsilon}, the eigenvalue of $\TFr_A$ on $H(\cB_A(\gamma,\lambda),D_A)$ is
\begin{equation}\label{eq:alpha}
\alpha(\gamma,\lambda) = p^{{\rm age}(\lambda)+n-1}\,\epsilon(\gamma,\lambda) = p^{{\rm age}(\lambda)-1}\,(-p)^{{\rm age}^\vee(\gamma)}\,\Gamma_p(\gamma A^{-1})\,.
\end{equation}
The supertrace sign in \eqref{eq:supertrace-nu} is $(-1)^{\dim(\lambda)}$, coming from the total degree $\dim(\lambda)+2\,{\rm age}(\lambda)-2$ assigned by the Hodge bigrading.  Since $s+r = \dim(\lambda)+2\,{\rm age}(\lambda)-2$, we have $(-1)^{s+r}=(-1)^{\dim(\lambda)}$ and \eqref{eq:supertrace-nu} can be rewritten as
\begin{equation}\label{eq:supertrace-sr}
{\rm ST}_{p^\nu}(A,G)=\sum_i (-1)^{s_i+r_i}\,\alpha_i^\nu\,,
\end{equation}
where the sum is over all basis elements with $\delta(\gamma,\lambda)=1$ and $\alpha_i=\alpha(\gamma_i,\lambda_i)$. At $\nu=1$, this recovers the supertrace formula of \cite[Theorem~3.6]{AP2}. The trace contribution ${\rm tc}_i$ from \cite{AP2} (that is, the summand of ${\rm ST}_{p^\nu}(A,G)$ labeled by $(\gamma_i,\lambda_i)$) is related to the eigenvalue by $\alpha_i = (-1)^{\dim(\lambda)}\,{\rm tc}_i = (-1)^{s_i+r_i}\,{\rm tc}_i$.
\end{remark}

\begin{definition}
Let $A$ be a Berglund--H\"ubsch matrix over $\mathbb F_p$ and let $G$ be a subgroup of $G_A$ such that $J=G\cap G^T$. We define the {\it Berglund--H\"ubsch zeta function} of the pair $(A,G)$ over $\mathbb F_p$ to be the generating function
\begin{equation}\label{eq:bh-zeta}
\zeta_p(A,G,t)= \exp\!\left(\sum_{\nu=1}^{\infty} {\rm ST}_{p^\nu}(A,G)\,\frac{t^\nu}{\nu}\right)\,.
\end{equation}
\end{definition}

\begin{remark}
We can use the Berglund--H\"ubsch zeta function to restate Conjecture \ref{conj:supertrace-counts-points} as follows. Let $A$ be an $n\times n$ Berglund--H\"ubsch matrix over $\mathbb F_p$ with $n\ge 3$ and let $G$ be a subgroup of $G_A$ such that $J\in G\cap G^T$. Then $\zeta_p(A,G,t)$ is equal to the local zeta function of any crepant resolution of $Y_{A,G}(\mathbb F_p)$.
\end{remark}

\begin{example}\label{ex:3.3}
While we will see in Example \ref{ex:L2L2} and Example \ref{ex:quintic} that ${\rm ST}_{p^\nu}(A,G)$ correctly calculates the number of points of $Z_{A,G}(\mathbb F_{p^\nu})$ outside of the $\det(A)\,|\, (p-1)$ range, the hypothesis cannot be dropped altogether. For instance, consider the Berglund--H\"ubsch matrix
\begin{equation}\label{eq:3.3}
A = \mat{2 & 1 & 0 \\ 0 & 2 & 1 \\ 0 & 0 & 3}
\end{equation}
so that $X_A$ defines a cubic curve in the projective plane. Accordingly, $JA^{-T}J^T=1$ and the Calabi--Yau condition is satisfied. The orbifold cohomology calculated with respect to $G=\langle J\rangle$ has $4$ basis elements. Table~\ref{tab:chain-223-traces} lists all trace contributions and eigenvalues.

\begin{table}[ht]
\centering
\caption{Orbifold trace contributions, with $A$ as in \eqref{eq:3.3} and $G=\langle J\rangle$.}\label{tab:chain-223-traces}
\renewcommand{\arraystretch}{1.1}
{\footnotesize
\begin{tabular}{@{}llll@{}}
\toprule
$H^{s,r}$ & Element & $\gamma A^{-1}$ & $\alpha_i$ \\
\midrule
$H^{0,0}$ & $y_1 y_2 y_3$ & $(0,0,0)$ & $1$ \\
\midrule
$H^{0,1}$ & $x_1 x_2 x_3\, e_1 e_2 e_3$ & $(\tfrac{1}{2},\tfrac{1}{4},\tfrac{1}{4})$ & $-\,\Gamma_p(\tfrac{1}{2})\,\Gamma_p(\tfrac{1}{4})\,\Gamma_p(\tfrac{1}{4})$ \\
\midrule
$H^{1,0}$ & $x_1 x_2^2 x_3^3\, e_1 e_2 e_3$ & $(\tfrac{1}{2},\tfrac{3}{4},\tfrac{3}{4})$ & $p\,\Gamma_p(\tfrac{1}{2})\,\Gamma_p(\tfrac{3}{4})\,\Gamma_p(\tfrac{3}{4})$ \\
\midrule
$H^{1,1}$ & $y_1^2 y_2^2 y_3^2$ & $(0,0,0)$ & $p$ \\
\bottomrule
\end{tabular}
}
\renewcommand{\arraystretch}{1.0}
\end{table}

If $p=7$, then the exponents of $W_A$ all divide $p-1=6$ but $\det(A)=12\nmid 6$. Brute force shows that $X_A(\mathbb F_7)$ has 8 points. On the other hand using \cite{P} we calculate
\[
{\rm ST}_p(A,\langle J\rangle)=6\cdot 7^2 + 7^3 + 5\cdot 7^4 + 5\cdot 7^5 + 5\cdot 7^6 + 4\cdot 7^7 + \cdots \in\Z_7,
\]
which is not an integer. On the other hand, if $p = 13$,  then $\det(A)=12\mid 12=p-1$ and ${\rm ST}_{13}(A,\langle J\rangle)=20$ matches the brute force calculation of $\#X_A(\mathbb F_{13})$. In this case, the Berglund--H\"ubsch zeta function
\[
\zeta_p(A,\langle J\rangle,t) = \frac{1+6t+pt^2}{(1-t)(1-pt)},
\]
matches the local zeta function of $X_A(\mathbb F_{13})$.
\end{example}

\begin{example}\label{ex:3.6}
Even though Conjecture \ref{conj:supertrace-counts-points} is motivated in part by the Crepant Resolution Conjecture and the latter is proved \cite{Yasuda04} in characteristic zero for Gorenstein varieties that are not necessarily Calabi--Yau, the condition $J\in G\cap G^T$ cannot be completely removed. For instance, consider
\begin{equation}\label{eq:non-cy-4chain}
A = \mat{2 & 1 & 0 & 0 \\ 0 & 2 & 1 & 0 \\ 0 & 0 & 2 & 1 \\ 0 & 0 & 0 & 3}
\end{equation}
and $G=\langle J\rangle$ with $|G|=3$ so that $X_A$ defines a smooth cubic in $\mathbb{P}^3$. Accordingly, the Calabi--Yau condition fails: $JA^{-T}J^T=\tfrac{4}{3}\notin\Z$.

The orbifold cohomology has $8$ basis elements. The $6$ elements from the untwisted sector $\lambda=0$ have integer ages and well-defined eigenvalues. The $2$ elements from the twisted sectors $\lambda=J,2J$ have fractional ages $\tfrac{4}{3}$ and $\tfrac{8}{3}$, so their eigenvalues involve $p^{1/3}\notin\Z_p$. Table~\ref{tab:chain-2223-traces} lists all trace contributions.

\begin{table}[ht]
\centering
\caption{Orbifold trace contributions for $A$ as in \eqref{eq:non-cy-4chain}. and $G=\langle J\rangle$}\label{tab:chain-2223-traces}
\renewcommand{\arraystretch}{1.1}
{\scriptsize
\begin{tabular}{@{}llll@{}}
\toprule
$H^{s,r}$ & Element & $\gamma A^{-1}$ & $\alpha_i$ \\
\midrule
$H^{1/3,1/3}$ & $y_1 y_2 y_3 y_4$ & $(0,0,0,0)$ & $p^{1/3}$ \\
\midrule
$H^{1,1}$ & $x_1 x_2 x_3 x_4^3\, e_1 e_2 e_3 e_4$ & $(\tfrac{1}{2},\tfrac{1}{4},\tfrac{3}{8},\tfrac{7}{8})$ & $p\,\Gamma_p(\tfrac{1}{2})\,\Gamma_p(\tfrac{1}{4})\,\Gamma_p(\tfrac{3}{8})\,\Gamma_p(\tfrac{7}{8})$ \\
 & $x_1 x_2 x_3^2 x_4^2\, e_1 e_2 e_3 e_4$ & $(\tfrac{1}{2},\tfrac{1}{4},\tfrac{7}{8},\tfrac{3}{8})$ & $p\,\Gamma_p(\tfrac{1}{2})\,\Gamma_p(\tfrac{1}{4})\,\Gamma_p(\tfrac{7}{8})\,\Gamma_p(\tfrac{3}{8})$ \\
 & $x_1 x_2^2 x_3 x_4^2\, e_1 e_2 e_3 e_4$ & $(\tfrac{1}{2},\tfrac{3}{4},\tfrac{1}{8},\tfrac{5}{8})$ & $p\,\Gamma_p(\tfrac{1}{2})\,\Gamma_p(\tfrac{3}{4})\,\Gamma_p(\tfrac{1}{8})\,\Gamma_p(\tfrac{5}{8})$ \\
 & $x_1 x_2^2 x_3^2 x_4\, e_1 e_2 e_3 e_4$ & $(\tfrac{1}{2},\tfrac{3}{4},\tfrac{5}{8},\tfrac{1}{8})$ & $p\,\Gamma_p(\tfrac{1}{2})\,\Gamma_p(\tfrac{3}{4})\,\Gamma_p(\tfrac{5}{8})\,\Gamma_p(\tfrac{1}{8})$ \\
 & $x_1^2 x_2 x_3 x_4^2\, e_1 e_2 e_3 e_4$ & $(1,0,\tfrac{1}{2},\tfrac{1}{2})$ & $p\,\Gamma_p(1)\,\Gamma_p(\tfrac{1}{2})^2$ \\
 & $x_1^2 x_2 x_3^2 x_4\, e_1 e_2 e_3 e_4$ & $(1,0,1,0)$ & $p\,\Gamma_p(1)^2$ \\
\midrule
$H^{5/3,5/3}$ & $y_1^2 y_2^2 y_3^2 y_4^2$ & $(0,0,0,0)$ & $p^{5/3}$ \\
\bottomrule
\end{tabular}
}
\renewcommand{\arraystretch}{1.0}
\end{table}

Let $p=73$. Here $\det(A)=24\mid 72=p-1$. All six products of $p$-adic gamma functions appearing in the contributions from $H^{1,1}$ evaluate to $p$, giving a total contribution of $N_{\mathrm{partial}} = 6\cdot 73 = 438$. The brute-force count over $\mathbb{P}^3(\F_{73})$ gives $N_{\mathrm{true}}=5841$, and
\[
N_{\mathrm{true}} - N_{\mathrm{partial}} = 5841 - 438 = 5403 = 1 + 73 + 73^2 = \#\mathbb{P}^2(\F_{73}).
\]
The difference is exactly $\#\mathbb{P}^2(\F_p)$, the contribution the fractional-age twisted sectors would supply if the Calabi--Yau condition held.
\end{example}

\begin{remark}\label{rem:section4-generalization}
In \cite{AP2}, Conjecture~\ref{conj:supertrace-counts-points} was established in the $\nu=1$ and $G=\langle J\rangle$ case for: 1) all Berglund--H\"ubsch elliptic curves, 2) all weighted diagonal hypersurfaces with coprime weights, and 3) all 14 weighted diagonal K3 surfaces. The proofs given there can be straightforwardly adapted to establish Conjecture~\ref{conj:supertrace-counts-points} in the $\nu>1$ case.
\end{remark}

\section{The Weil Conjectures for $\zeta_p(A,G,t)$}

Since the Weil conjectures hold for any crepant resolution of $Y_{A,G}$, the following offers a non-trivial check for Conjecture~\ref{conj:supertrace-counts-points}:

\begin{theorem}\label{thm:3.6}
The Berglund--H\"ubsch zeta function formally satisfies the Weil conjectures. More precisely, let $A$ be an $n\times n$ Berglund--H\"ubsch matrix over $\mathbb F_p$ with $n\ge 3$ and let $G$ be a subgroup of $G_A$ such that $J\in G\cap G^T$. Then
\begin{enumerate}[1)]
\item $\zeta_p(A,G,t)$ is a rational function of $t$;
\item $\zeta_p(A,G,1/(p^{n-2}t))=\pm p^{\frac{(n-2)\chi}{2}}t^\chi \zeta_p(A,G,t)$ where $\chi$ denotes the Euler characteristic of any crepant resolution of $Y_{A,G}(\mathbb C)$.
\item For each $\alpha_i$ as in \eqref{eq:supertrace-sr}, $|\alpha_i|=p^{(s_i+r_i)/2}$, where $|\cdot|$ denotes the standard Euclidean absolute value.
\end{enumerate}
\end{theorem}

\begin{proof}

If we let $\alpha_i$ be as in \eqref{eq:supertrace-sr}, then using the identity $\exp\!\left(-\sum_{\nu\ge 1}\frac{(\alpha t)^\nu}{\nu}\right)=1-\alpha t$, we obtain the factorization
\begin{equation}\label{eq:zeta-factored}
\zeta_p(A,G,t) = \prod_i (1-\alpha_i\,t)^{(-1)^{s_i+r_i+1}} = \frac{\displaystyle\prod_{s_i+r_i\text{ odd}}(1-\alpha_i\,t)}{\displaystyle\prod_{s_i+r_i\text{ even}}(1-\alpha_i\,t)}=\prod_k P_k(t)^{(-1)^{k+1}}\,.
\end{equation}
where $P_k(t)=\prod_{s_i+r_i=k}(1-\alpha_i\,t)$ keeps track of all terms corresponding to total degree $k$. This proves 1).

Recall from \cite[Remark~2.18]{AP2} that each $\alpha_i = \alpha(\gamma,\lambda)$ sits in total Hodge degree $s + r$, where $s = {\rm age}(\lambda)+{\rm age}^\vee(\gamma)-1$ and $r = \dim(\lambda)+{\rm age}(\lambda)-{\rm age}^\vee(\gamma)-1$. It follows from the Gross-Koblitz formula (see e.g.\ \cite{Koblitz84}) that $\Gamma_p(a/(p-1))$ can be thought of as an algebraic number and
\begin{equation}
\left|\Gamma_p\left(\frac{a}{p-1}\right)\right|=p^{\frac{1}{2}-\frac{a}{p-1}}
\end{equation}
for any integer $0<a<p-1$. Hence
\begin{equation}\label{eq:3.6}
\left|\Gamma_p(\gamma A^{-1})\right|=p^{\frac{\dim(\lambda)}{2}-{\rm age}^\vee(\gamma)}
\end{equation}
and, upon substitution of \eqref{eq:3.6} into \eqref{eq:alpha}, we obtain 3).

It remains to prove 2). As shown in \cite{AP2}, the bigraded components of the cohomology of $(\mathcal B^G_A,D_A)$ are isomorphic to those of the corresponding components of the Chen-Ruan cohomology of $Y_{A,G}(\mathbb C)$, and thus, by the proof of the crepant resolution conjecture \cite{Yasuda04}, to those of any crepant resolution $Z_{A,G}(\mathbb C)$ of $Y_{A,G}(\mathbb C)$. We then have that $\chi = \sum_i (-1)^{s_i + r_i}$, and by factoring out $-\alpha_i / (p^{n-2} t)$ from each term in the expression for $\zeta_p(A, G, 1/(p^{n-2}t))$, we arrive at:
\begin{equation}\label{eq:3.5}
\zeta_p(A,G,1/(p^{n - 2}t))= \pm p^{(n-2)\chi} t^\chi \prod_{i} \alpha_i^{(-1)^{s_i+r_i+1}} \prod_i \left(1-\frac{p^{n-2}t}{\alpha_i}\right)^{(-1)^{s_i+r_i+1}}\,.
\end{equation}

Furthermore, we know from \cite{AP} that the cohomology of $(\mathcal B_A^G,D_A)$ breaks down as the direct sum of components labeled by $\lambda\in G$, with each such component isomorphic to the $G$-invariant part of the Milnor ring of the potentials $W_{A^\lambda}$. The non-degenerate residue pairing on the Milnor ring of $W_{A^\lambda}$ pairs  $\gamma$ with $J A^\lambda_0-\gamma$ where $A^\lambda_0$ denotes the $n\times n$ matrix obtained by first restricting to $A^\lambda$ and then extending back by setting all other entries to $0$. Moreover, $\gamma$ is in the $G$-invariant part (i.e., it has integer dual age) if and only if $J A^\lambda_0-\gamma$ is also. Hence $\delta(\gamma,\lambda)=1$ if and only if $\delta(J A^\lambda_0-\gamma,\lambda)=1$. Reversing the role of $\gamma$ and $\lambda$ (which geometrically corresponds to the perspective of the mirror Calabi--Yau orbifold $Y_{A^T,G^T}$) and repeating the same argument, we conclude that $\delta(J A^\lambda_0-\gamma,J (A^T)^\gamma_0-\lambda)=1$ if and only if $\delta(\gamma,\lambda)=1$.

Using the shorthand $\alpha_i = \alpha(\gamma, \lambda)$ and $\alpha_{i'} = \alpha(J A^\lambda_0-\gamma, J(A^T)^\gamma_0-\lambda)$, we now show that
\begin{equation}\label{eq:weil-alpha-alpha-prime}
\alpha_i \alpha_{i'} = \alpha(\gamma,\lambda)\alpha(J A^\lambda_0-\gamma, J(A^T)^\gamma_0-\lambda)=p^{n-2}\,.
\end{equation}
Since multiplying $JA_0^\lambda$ by $A^{-1}$ produces an indicator vector for the $\lambda$ dimensions, we have that ${\rm age}^\vee(J A^\lambda_0) = \dim(\lambda)$, and symmetrically ${\rm age}(J(A^T)^\gamma_0) = \dim(\gamma)$. Therefore,
\begin{equation*}
{\rm age}(J (A^T)^\gamma_0)+{\rm age}^\vee (J A^\lambda_0)=\dim(\gamma)+\dim(\lambda)=n\,,
\end{equation*}
which used with \eqref{eq:alpha} produces
\begin{equation}
\alpha_{i'}= p^{n-1-{\rm age}(\lambda)-{\rm age}^\vee(\gamma)}(-1)^{\dim(\lambda)-{\rm age}^\vee(\gamma)}\Gamma_p((J A^\lambda_0-\gamma)A^{-1})\,.
\end{equation}
Together with the functional equation
\begin{equation}
\Gamma_p(\gamma A^{-1})\Gamma_p((J A^\lambda_0-\gamma)A^{-1})=(-1)^{p\dim(\lambda)-(p-1){\rm age}^\vee(\gamma)}
\end{equation}
and the fact that our assumptions on $A$ imply that $p$ is odd, this yields \eqref{eq:weil-alpha-alpha-prime}.

Since $\dim(\lambda)=\dim(J (A^T)^\gamma_0-\lambda)$, we see that $(\gamma,\lambda)$ and $(J A^\lambda_0-\gamma, J (A^T)^\gamma_0-\lambda)$ have the same total Hodge degree. Combined with \eqref{eq:weil-alpha-alpha-prime} this shows that
\begin{equation}\label{eq:weil-last-term}
 \prod_i \left(1-\frac{p^{n-2}t}{\alpha_i}\right)^{(-1)^{s_i+r_i+1}}
 = \prod_{i'} (1 - \alpha_{i'} t)^{(-1)^{s_{i'} + r_{i'} + 1}}
 = \zeta_p(A,G,t)\,.
\end{equation}
Substituting \eqref{eq:weil-last-term} into \eqref{eq:3.5} and observing that, by 3),
\begin{equation}
\left|\prod_i \alpha_i^{(-1)^{s_i+r_i+1}}\right|=p^{-\frac{(n-2)\chi}{2}}\,,
\end{equation}
we arrive at
\begin{equation}\label{eq:weil-middle-term}
\zeta_p(A,G,1/(p^{n-2}t))= \pm p^{\frac{(n-2)\chi}{2}} t^\chi C \zeta_p(A,G,t)
\end{equation}
for some $C$ such that $|C|=1$. However, applying \eqref{eq:weil-middle-term} twice straightforwardly implies $C^2=1$, from which 2) easily follows.
\end{proof}

\begin{remark}
Theorem \ref{thm:3.6} only establishes that $\zeta_p(A,G,t)$ is a rational function with coefficients in $\mathbb Q_p$. However, Conjecture~\ref{conj:supertrace-counts-points} implies the stronger statement that $\zeta_p(A,G,t)$ is a rational function with coefficients in $\mathbb Q$.
\end{remark}

\section{K3 Surfaces}\label{sec:k3-examples}

Up to isomorphism, there are exactly $14$ families of weighted diagonal surfaces whose minimal resolutions are K3 surfaces. In \cite{Goto}, these families are determined by their degree $m$, weights $Q = (q_0, q_1, q_2, q_3)$, and a ``twist'' $c = (c_0, c_1, c_2, c_3)$ corresponding to the monomial coefficients of the polynomial. A rescaling of the coordinates reduces the polynomial to the untwisted case $c = (1, 1, 1, 1)$, which is the setting of our paper. In our framework, we view each untwisted surface as being associated with the Berglund-H\"{u}bsch matrix $A = \mathrm{diag}(m/q_0, \ldots, m/q_3)$ and symmetry group $G = \langle J \rangle$. The resulting K3 surface is the minimal resolution of the variety cut out by the potential
\[
x_0^{m/q_0} + x_1^{m/q_1} + x_2^{m/q_2} + x_3^{m/q_3} = 0\,
\]
in the weighted projective space $\mathbb{P}^3(Q)$.

All but two of these K3s (the Fermat quartic and $x_0^2+x_1^6+x_2^6+x_3^6=0$) possess isolated cyclic quotient singularities at loci $\mathcal{P}_{ij}$ consisting of points where exactly two coordinates are non-zero and $\gcd(q_i, q_j) \geq 2$. Goto \cite[Theorem~5.2]{Goto} explicitly computed the minimal resolutions and the corresponding zeta functions in all 14 cases.
In the untwisted case, there are only two such K3s where the minimal field of definition of the singularities can strictly require the quadratic extension $\F_{p^2}$: the surface $x_0^2 + x_1^3 + x_2^{10} + x_3^{15} = 0$ considered in Example~\ref{ex:30-15-10-3-2}, and $x_0^3 + x_1^4 + x_2^4 + x_3^6 = 0$. However, the standing assumption $\det(A) \mid (p-1)$ implies that $(p - 1) / \gcd(m, p - 1)$ is even. Thus, $-1$ is an $m$-th power in $\F_p^\times$ and, by \cite[Section~4]{Goto}, this implies all singularities are defined over $\F_p$.

As observed in Remark \ref{rem:section4-generalization}, Conjecture \ref{conj:supertrace-counts-points} holds for weighted diagonal K3 surfaces when $G=\langle J\rangle$. Here we use the code \cite{P} to explicitly verify Conjecture~\ref{conj:supertrace-counts-points} for these $14$ untwisted surfaces over various primes. Table~\ref{tab:k3-zeta} lists the middle cohomology polynomial $P_2(t)$ of the zeta function computed over $\F_p$ for each surface at a prime $p$ satisfying $\det(A) \mid (p-1)$. For these K3 surfaces, the remaining factors are always $P_0(t) = 1-t$ and $P_4(t) = 1-p^2t$, and there is no odd-degree cohomology. In each case, $\zeta_p(A,\langle J\rangle,t)$ agrees exactly with the zeta function derived in \cite{Goto}. Example~\ref{ex:30-15-10-3-2} works through one case in detail.

{\small
\setlength{\LTpre}{0.5em}
\setlength{\LTpost}{0.5em}
\renewcommand{\arraystretch}{1.15}
\begin{longtable}{@{}rlr@{}}
\caption{Middle cohomology polynomials $P_2(t)$ for the $14$ diagonal K3 surfaces over $\F_p$, calculated using \cite{P}.}\label{tab:k3-zeta}\\
\toprule
$m$ & Equation & $p$ \\
\midrule
\endfirsthead
\toprule
$m$ & Equation & $p$ \\
\midrule
\endhead
\midrule
\multicolumn{3}{r@{}}{\footnotesize\itshape continued on next page}\\
\endfoot
\bottomrule
\endlastfoot
$42$ & $x_0^2+x_1^3+x_2^7+x_3^{42}=0$ in $\mathbb P^3(21,14,6,1)$ & $3529$ \\
\addlinespace[4pt]
\multicolumn{3}{c}{$\begin{aligned}[t]
P_2(t) ={} & (1-pt)^{10} \cdot\bigl(1-4675t+13699pt^2-16339p^2t^3+26774p^3t^4-25606p^4t^5 \\
& \qquad {}+33125p^5t^6-25606p^6t^7+26774p^7t^8-16339p^8t^9+13699p^9t^{10}-4675p^{10}t^{11}+p^{12}t^{12}\bigr)
\end{aligned}$} \\
\midrule
$30$ & $x_0^2+x_1^3+x_2^{10}+x_3^{15}=0$ in $\mathbb P^3(15,10,3,2)$ & $1801$ \\
\addlinespace[4pt]
\multicolumn{3}{c}{$\begin{aligned}[t]
P_2(t) ={} & (1-pt)^{14} \cdot\bigl(1-1873t+4068pt^2-5981p^2t^3+4595p^3t^4 \\
& \qquad {}-5981p^4t^5+4068p^5t^6-1873p^6t^7+p^8t^8\bigr)
\end{aligned}$} \\
\midrule
$24$ & $x_0^2+x_1^3+x_2^8+x_3^{24}=0$ in $\mathbb P^3(12,8,3,1)$ & $1153$ \\
\addlinespace[4pt]
\multicolumn{3}{c}{$\begin{aligned}[t]
P_2(t) ={} & (1-pt)^{10}(1+p^2t^2)^2 \cdot\bigl(1-2912t+2080pt^2+2464p^2t^3-5246p^3t^4 \\
& \qquad {}+2464p^4t^5+2080p^5t^6-2912p^6t^7+p^8t^8\bigr)
\end{aligned}$} \\
\midrule
$20$ & $x_0^2+x_1^4+x_2^5+x_3^{20}=0$ in $\mathbb P^3(10,5,4,1)$ & $1601$ \\
\addlinespace[4pt]
\multicolumn{3}{c}{$\begin{aligned}[t]
P_2(t) ={} & (1-pt)^{14} \cdot\bigl(1+6232t+13148pt^2+19304p^2t^3+21830p^3t^4 \\
& \qquad {}+19304p^4t^5+13148p^5t^6+6232p^6t^7+p^8t^8\bigr)
\end{aligned}$} \\
\midrule
$18$ & $x_0^2+x_1^3+x_2^9+x_3^{18}=0$ in $\mathbb P^3(9,6,2,1)$ & $2917$ \\
\addlinespace[4pt]
\multicolumn{3}{c}{$P_2(t) = (1-pt)^{16}(1+4530t+987pt^2-1316p^2t^3+987p^3t^4+4530p^4t^5+p^6t^6)$} \\
\midrule
$12$ & $x_0^2+x_1^3+x_2^{12}+x_3^{12}=0$ in $\mathbb P^3(6,4,1,1)$ & $2593$ \\
\addlinespace[4pt]
\multicolumn{3}{c}{$P_2(t) = (1-pt)^6(1+pt)^8(1+pt+p^2t^2)^2(1+850t-3405pt^2+850p^2t^3+p^4t^4)$} \\
\midrule
$12$ & $x_0^2+x_1^4+x_2^6+x_3^{12}=0$ in $\mathbb P^3(6,3,2,1)$ & $577$ \\
\addlinespace[4pt]
\multicolumn{3}{c}{$P_2(t) = (1-pt)^{18}(1+92t+966pt^2+92p^2t^3+p^4t^4)$} \\
\midrule
$12$ & $x_0^3+x_1^3+x_2^4+x_3^{12}=0$ in $\mathbb P^3(4,4,3,1)$ & $433$ \\
\addlinespace[4pt]
\multicolumn{3}{c}{$P_2(t) = (1-pt)^{18}(1+68t+294pt^2+68p^2t^3+p^4t^4)$} \\
\midrule
$12$ & $x_0^3+x_1^4+x_2^4+x_3^6=0$ in $\mathbb P^3(4,3,3,2)$ & $577$ \\
\addlinespace[4pt]
\multicolumn{3}{c}{$P_2(t) = (1-pt)^{18}(1+92t+966pt^2+92p^2t^3+p^4t^4)$} \\
\midrule
$10$ & $x_0^2+x_1^5+x_2^5+x_3^{10}=0$ in $\mathbb P^3(5,2,2,1)$ & $3001$ \\
\addlinespace[4pt]
\multicolumn{3}{c}{$P_2(t) = (1-pt)^{10}(1-9799t+14001pt^2-9799p^2t^3+p^4t^4)(1+pt+p^2t^2+p^3t^3+p^4t^4)^2$} \\
\midrule
$8$ & $x_0^2+x_1^4+x_2^8+x_3^8=0$ in $\mathbb P^3(4,2,1,1)$ & $7681$ \\
\addlinespace[4pt]
\multicolumn{3}{c}{$P_2(t) = (1+pt)^4(1-pt)^{14}(1-4300t-5466pt^2-4300p^2t^3+p^4t^4)$} \\
\midrule
$6$ & $x_0^3+x_1^3+x_2^6+x_3^6=0$ in $\mathbb P^3(2,2,1,1)$ & $1297$ \\
\addlinespace[4pt]
\multicolumn{3}{c}{$P_2(t) = (1-pt)^{16}(1+478t+p^2t^2)(1+pt+p^2t^2)^2$} \\
\midrule
$6$ & $x_0^2+x_1^6+x_2^6+x_3^6=0$ in $\mathbb P^3(3,1,1,1)$ & $433$ \\
\addlinespace[4pt]
\multicolumn{3}{c}{$P_2(t) = (1-pt)^{20}(1+862t+p^2t^2)$} \\
\midrule
$4$ & $x_0^4+x_1^4+x_2^4+x_3^4=0$ in $\mathbb P^3(1,1,1,1)$ & $257$ \\
\addlinespace[4pt]
\multicolumn{3}{c}{$P_2(t) = (1-pt)^{20}(1+510t+p^2t^2)$} \\
\end{longtable}
\renewcommand{\arraystretch}{1.0}
}

\begin{example}\label{ex:30-15-10-3-2}
Consider the Berglund--H\"ubsch matrix $A = \mathrm{diag}(2, 3, 10, 15)$ over the finite field $k = \F_p$ so that $\det(A)=900$, $m=30$. As symmetry group we choose $G=\langle J\rangle$. Denote by $\widetilde{X}_{\F_{p^\nu}}$ the minimal resolution of $W_A(x) = 0$ defined over the field $\F_{p^\nu}$. The singularities for this surface are summarized in Table~\ref{tab:sing-m30}.

It follows from \cite[Theorem~5.2]{Goto} that the middle cohomology polynomial of $\widetilde{X}_{\overline{\F}_{p}}$ can be expressed as
\begin{equation}
P_2(\widetilde{X}_{\overline{\F}_{p}}, t) =
(1 - p t) \,
\prod_{\mathbf{a} \in \mathfrak{A}} (1 - \mathcal{J}(\mathbf{a})\,t)
\prod_{(i,j) \in \mathscr{I}_0} \bigl\{(1 - p t)\cdots(1-(\eta_{ij}^{f_{ij}-1}p) t)\bigr\}^{r_{ij}\omega_{ij}} \,, \label{eq:goto-p2}
\end{equation}
where the notation follows \cite[Sections~4--5]{Goto}. The set $\mathfrak{A}$ consists of tuples $\mathbf{a} = (a_0, a_1, a_2, a_3)$ with each $a_i$ a nonzero multiple of~$q_i$ modulo~$m$ satisfying $\sum a_i \equiv 0 \pmod{m}$, and $\mathcal{J}(\mathbf{a}) = \mathcal{J}(\mathbf{1}, \mathbf{a})$ denotes the corresponding (untwisted) Jacobi sum. The set $\mathscr{I}_0 = \{(i,j) \mid 0 \le i < j \le 3,\; d_{ij} \ge 2\}$ indexes singular coordinate loci~$\mathcal{P}_{ij}$ from Table~\ref{tab:sing-m30}, where $d_{ij} = \gcd(q_i, q_j)$. Each point of $\mathcal{P}_{ij}$ is a cyclic quotient singularity of type $A_{d_{ij}, \alpha_{ij}}$, where $\alpha_{ij}$ is the unique integer in $[1, d_{ij})$ satisfying $q_{i_*}\alpha_{ij} \equiv q_{j_*} \pmod{d_{ij}}$ for $\{i_*, j_*\} = \{0,1,2,3\} \setminus \{i,j\}$. The exponent $r_{ij}$ is the number of exceptional curves in the minimal resolution, equal to the length of the Hirzebruch--Jung continued fraction of $d_{ij}/\alpha_{ij}$. The parameter $\omega_{ij} = m / (e_{ij} f_{ij})$ counts Galois orbits in $\mathcal{P}_{ij}$, where $e_{ij} = \operatorname{lcm}(q_i, q_j)$ and $f_{ij}$ is the degree of the field of definition over~$\F_p$. Finally, $\eta_{ij}$ is a primitive $f_{ij}$-th root of unity. Under the standing assumption $\det(A) \mid (p-1)$, we have $f_{ij} = 1$ for all $(i,j)$, so the singularity product simplifies to $(1-pt)^{r_{ij}\omega_{ij}}$. The Jacobi sums are listed in Table~\ref{tab:jacobi-m30}. Substituting into~\eqref{eq:goto-p2} yields the middle cohomology polynomial already listed in Table~\ref{tab:k3-zeta} in the row determined by $m = 30$.

\begin{table}[ht]
\centering
\begin{tabular}{ccccc}
\toprule
Locus $\mathcal{P}_{ij}$ & Type & Multiplicity & Exceptional curves ($r_{ij}$) & Sing.\ contribution ($\nu{=}1$) \\
\midrule
$\mathcal{P}_{01}$ & $A_{5,4}$ & 1 & 4 & $(1 - pt)^{4}$ \\
$\mathcal{P}_{02}$ & $A_{3,2}$ & 2 & 2 & $(1 - pt)^{4}$ \\
$\mathcal{P}_{13}$ & $A_{2,1}$ & 3 & 1 & $(1 - pt)^{3}$ \\
\bottomrule
\end{tabular}
\caption{Singularities of the diagonal K3 surface with $m = 30$ and $Q = (15, 10, 3, 2)$; see \cite[Lemma~4.5 and Proposition~4.6]{Goto}.}
\label{tab:sing-m30}
\end{table}

\begin{table}[ht]
\centering
\caption{Jacobi sums $\mathcal{J}(\mathbf{a})$ for $m=30$, $Q=(15,10,3,2)$ at $p=1801$.}\label{tab:jacobi-m30}
\renewcommand{\arraystretch}{1.2}
\begin{tabular}{@{}c@{\qquad}c@{}}
\begin{tabular}[t]{@{}lc@{}}
\toprule
$\mathbf{a}$ & $\mathcal{J}(\mathbf{a})$ \\
\midrule
$(\tfrac{1}{2}, \tfrac{1}{3}, \tfrac{1}{10}, \tfrac{1}{15})$ & $1775-307i$ \\
$(\tfrac{1}{2}, \tfrac{1}{3}, \tfrac{3}{10}, \tfrac{13}{15})$ & $-215+1788i$ \\
$(\tfrac{1}{2}, \tfrac{1}{3}, \tfrac{1}{2}, \tfrac{2}{3})$ & $p$ \\
$(\tfrac{1}{2}, \tfrac{1}{3}, \tfrac{7}{10}, \tfrac{7}{15})$ & $-706-1657i$ \\
$(\tfrac{1}{2}, \tfrac{1}{3}, \tfrac{9}{10}, \tfrac{4}{15})$ & $83+1799i$ \\
\bottomrule
\end{tabular}
&
\begin{tabular}[t]{@{}lc@{}}
\toprule
$\mathbf{a}$ & $\mathcal{J}(\mathbf{a})$ \\
\midrule
$(\tfrac{1}{2}, \tfrac{2}{3}, \tfrac{1}{10}, \tfrac{11}{15})$ & $83-1799i$ \\
$(\tfrac{1}{2}, \tfrac{2}{3}, \tfrac{3}{10}, \tfrac{8}{15})$ & $-706+1657i$ \\
$(\tfrac{1}{2}, \tfrac{2}{3}, \tfrac{1}{2}, \tfrac{1}{3})$ & $p$ \\
$(\tfrac{1}{2}, \tfrac{2}{3}, \tfrac{7}{10}, \tfrac{2}{15})$ & $-215-1788i$ \\
$(\tfrac{1}{2}, \tfrac{2}{3}, \tfrac{9}{10}, \tfrac{14}{15})$ & $1775+307i$ \\
\bottomrule
\end{tabular}
\end{tabular}
\renewcommand{\arraystretch}{1.0}
\end{table}

The orbifold cohomology has $24$ basis elements with a Hodge diamond of $h^{0,0}=h^{0,2}=h^{2,0}=h^{2,2}=1$ and $h^{1,1}=20$.
Since the cohomology is concentrated in even degrees, the supertrace signs are all positive, yielding $\alpha_i = \mathrm{tc}_i$. Table~\ref{tab:eigenvalue-comparison} lists all $24$ eigenvalues. Note that while some of the individual eigenvalues lie in $\Q_p \setminus \Q$, substituting them into \eqref{eq:zeta-factored} yields a zeta function with coefficients in $\Z$ that matches the one calculated using Goto's framework.

\begin{table}[ht]
\centering
\caption{Twisted Frobenius eigenvalues for $x_0^2 + x_1^3 + x_2^{10} + x_3^{15} = 0$.}\label{tab:eigenvalue-comparison}
\renewcommand{\arraystretch}{1.1}
{\scriptsize
\begin{tabular}{@{}lllc@{}}
\toprule
$H^{s,r}$ & Element & $\alpha_i = \mathrm{tc}_i$ & $\mathrm{tc}_i$ at $p=1801$ \\
\midrule
$H^{0,0}$ & $y_1 y_2 y_3 y_4$ & $1$ & $1$ \\
\midrule
$H^{0,2}$ & $x_1 x_2 x_3 x_4\,e_1 e_2 e_3 e_4$ & $-\Gamma_p(\tfrac{1}{2})\Gamma_p(\tfrac{1}{3})\Gamma_p(\tfrac{1}{10})\Gamma_p(\tfrac{1}{15})$ & $72{+}845{\cdot}p{+}1550{\cdot}p^{2}{+}721{\cdot}p^{3}{+}\cdots$ \\
\midrule
$H^{1,1}$ & $y_1 y_2 y_3^{3} y_4^{13}$ & $p$ & $p$ \\
 & $y_1 y_2 y_3^{5} y_4^{10}$ & $p$ & $p$ \\
 & $y_1 y_2 y_3^{7} y_4^{7}$ & $p$ & $p$ \\
 & $y_1 y_2 y_3^{9} y_4^{4}$ & $p$ & $p$ \\
 & $y_1 y_2^{2} y_3 y_4^{11}$ & $p$ & $p$ \\
 & $y_1 y_2^{2} y_3^{3} y_4^{8}$ & $p$ & $p$ \\
 & $y_1 y_2^{2} y_3^{5} y_4^{5}$ & $p$ & $p$ \\
 & $y_1 y_2^{2} y_3^{7} y_4^{2}$ & $p$ & $p$ \\
 & $x_2 x_4^{10} y_1 y_3^{5}\,e_2 e_4$ & $-p\,\Gamma_p(\tfrac{1}{3})\Gamma_p(\tfrac{2}{3})$ & $p$ \\
 & $x_2^{2} x_4^{5} y_1 y_3^{5}\,e_2 e_4$ & $-p\,\Gamma_p(\tfrac{2}{3})\Gamma_p(\tfrac{1}{3})$ & $p$ \\
 & $x_1 x_3^{5} y_2 y_4^{10}\,e_1 e_3$ & $-p\,\Gamma_p(\tfrac{1}{2})^2$ & $p$ \\
 & $x_1 x_3^{5} y_2^{2} y_4^{5}\,e_1 e_3$ & $-p\,\Gamma_p(\tfrac{1}{2})^2$ & $p$ \\
 & $x_1 x_2 x_3^{3} x_4^{13}\,e_1 e_2 e_3 e_4$ & $p\,\Gamma_p(\tfrac{1}{2})\Gamma_p(\tfrac{1}{3})\Gamma_p(\tfrac{3}{10})\Gamma_p(\tfrac{13}{15})$ & $1416{\cdot}p{+}391{\cdot}p^{2}{+}37{\cdot}p^{3}{+}\cdots$ \\
 & $x_1 x_2 x_3^{5} x_4^{10}\,e_1 e_2 e_3 e_4$ & $p\,\Gamma_p(\tfrac{1}{2})\Gamma_p(\tfrac{1}{3})\Gamma_p(\tfrac{1}{2})\Gamma_p(\tfrac{2}{3})$ & $p$ \\
 & $x_1 x_2 x_3^{7} x_4^{7}\,e_1 e_2 e_3 e_4$ & $p\,\Gamma_p(\tfrac{1}{2})\Gamma_p(\tfrac{1}{3})\Gamma_p(\tfrac{7}{10})\Gamma_p(\tfrac{7}{15})$ & $1081{\cdot}p{+}438{\cdot}p^{2}{+}783{\cdot}p^{3}{+}\cdots$ \\
 & $x_1 x_2 x_3^{9} x_4^{4}\,e_1 e_2 e_3 e_4$ & $p\,\Gamma_p(\tfrac{1}{2})\Gamma_p(\tfrac{1}{3})\Gamma_p(\tfrac{9}{10})\Gamma_p(\tfrac{4}{15})$ & $409{\cdot}p{+}1590{\cdot}p^{2}{+}1437{\cdot}p^{3}{+}\cdots$ \\
 & $x_1 x_2^{2} x_3 x_4^{11}\,e_1 e_2 e_3 e_4$ & $p\,\Gamma_p(\tfrac{1}{2})\Gamma_p(\tfrac{2}{3})\Gamma_p(\tfrac{1}{10})\Gamma_p(\tfrac{11}{15})$ & $1026{\cdot}p{+}1383{\cdot}p^{2}{+}580{\cdot}p^{3}{+}\cdots$ \\
 & $x_1 x_2^{2} x_3^{3} x_4^{8}\,e_1 e_2 e_3 e_4$ & $p\,\Gamma_p(\tfrac{1}{2})\Gamma_p(\tfrac{2}{3})\Gamma_p(\tfrac{3}{10})\Gamma_p(\tfrac{8}{15})$ & $903{\cdot}p{+}410{\cdot}p^{2}{+}194{\cdot}p^{3}{+}\cdots$ \\
 & $x_1 x_2^{2} x_3^{5} x_4^{5}\,e_1 e_2 e_3 e_4$ & $p\,\Gamma_p(\tfrac{1}{2})\Gamma_p(\tfrac{2}{3})\Gamma_p(\tfrac{1}{2})\Gamma_p(\tfrac{1}{3})$ & $p$ \\
 & $x_1 x_2^{2} x_3^{7} x_4^{2}\,e_1 e_2 e_3 e_4$ & $p\,\Gamma_p(\tfrac{1}{2})\Gamma_p(\tfrac{2}{3})\Gamma_p(\tfrac{7}{10})\Gamma_p(\tfrac{2}{15})$ & $1525{\cdot}p{+}1463{\cdot}p^{2}{+}303{\cdot}p^{3}{+}\cdots$ \\
\midrule
$H^{2,0}$ & $x_1 x_2^{2} x_3^{9} x_4^{14}\,e_1 e_2 e_3 e_4$ & $-p^2\,\Gamma_p(\tfrac{1}{2})\Gamma_p(\tfrac{2}{3})\Gamma_p(\tfrac{9}{10})\Gamma_p(\tfrac{14}{15})$ & $1776{\cdot}p^{2}{+}1343{\cdot}p^{3}{+}473{\cdot}p^{4}{+}\cdots$ \\
\midrule
$H^{2,2}$ & $y_1 y_2^{2} y_3^{9} y_4^{14}$ & $p^2$ & $p^2$ \\
\bottomrule
\end{tabular}
}
\renewcommand{\arraystretch}{1.0}
\end{table}
\end{example}

Next, we examine our conjecture for non-diagonal K3 surfaces. Goto computed zeta functions for so-called \emph{deformed diagonal} K3 surfaces, which are defined as the minimal resolutions of surfaces defined by equations of the form
\[
x_0^{a_0} + x_1^{a_1} + x_2^{a_2} + x_0 x_3^{a_3 - 1} = 0 \quad\text{in}\;\mathbb P^3(q_0, q_1, q_2, q_3)
\]
where $a_i = m/q_i$ and $\sum q_i = m$. All $85$ deformed diagonal surfaces are listed in \cite[Table~10]{GotoThesis}, along with a description of their singularities, their resolutions, and zeta functions in Theorem~4.3.4. We will examine a few examples with varying singularity structures.

\begin{example}\label{ex:deformed-diagonal-37}
Consider the deformed diagonal surface
\[
x_0^5 + x_1^{15} + x_2^3 + x_0 x_3^2 = 0 \quad\text{in}\;\mathbb P^3(3,1,5,6),
\]
with
\[
A = \begin{pmatrix} 5 & 0 & 0 & 0 \\ 0 & 15 & 0 & 0 \\ 0 & 0 & 3 & 0 \\ 1 & 0 & 0 & 2 \end{pmatrix},
\]
$\det(A) = 450$, $M = 30$, and $G = \langle J\rangle$. The singularities and corresponding contributions to the zeta function over $\F_{1801}, \F_{2251}$ and $\F_{2251^2}$ are summarized in Table~\ref{tab:sing-dd37}.

\begin{table}[ht]
\centering
\renewcommand{\arraystretch}{1.15}
{\small
\begin{tabular}{cccccccccccc}
\toprule
$p$ & $\nu$ & Locus & $m - q_0$ & $d_{ij}$ & $\alpha_{ij}$ & $e_{ij}$ & $m_{ij}$ & $f_{ij}$ & $\omega_{ij}$ & $r_{ij}$ & Contribution \\
\midrule
$1801$ & $1$ & $[0{:}0{:}0{:}1]$ & --- & --- & --- & --- & --- & --- & --- & $5$ & $(1-pt)^5$ \\
$1801$ & $1$ & $\mathcal{P}_{03}$ & $12$ & $3$ & $2$ & $6$ & $1$ & $1$ & $2$ & $2$ & $(1-pt)^4$ \\
\midrule
$2251$ & $1$ & $[0{:}0{:}0{:}1]$ & --- & --- & --- & --- & --- & --- & --- & $5$ & $(1-pt)^5$ \\
$2251$ & $1$ & $\mathcal{P}_{03}$ & $12$ & $3$ & $2$ & $6$ & $2$ & $2$ & $1$ & $2$ & $(1-pt)^2(1+pt)^2$ \\
\midrule
$2251$ & $2$ & $[0{:}0{:}0{:}1]$ & --- & --- & --- & --- & --- & --- & --- & $5$ & $(1-qt)^5$ \\
$2251$ & $2$ & $\mathcal{P}_{03}$ & $12$ & $3$ & $2$ & $6$ & $2$ & $2$ & $1$ & $2$ & $(1-qt)^4$ \\
\bottomrule
\end{tabular}
}
\renewcommand{\arraystretch}{1.0}
\caption{Singularities of the deformed diagonal surface with $Q=(3,1,5,6)$ and $m=15$. Shown are the parameters appearing in the contribution of the singularity to the zeta function, as described in \cite[Theorem~4.3.4]{GotoThesis}. Here $q = p^2$.}\label{tab:sing-dd37}
\end{table}

The singularity~$\mathcal{P}_{03}$ is defined only over $\F_{2251^2}$ and not $\F_{2251}$, which is reflected in the contribution to the zeta function when defining the surface over $\F_{2251^2}$ vs $\F_{2251}$. By \cite[Theorem~4.3.4]{GotoThesis}, there are two $\mathfrak{V}$-terms and $10$ $\mathfrak{W}$-terms ($8$ with Jacobi sums with values in $\C \setminus \R$ for both primes). Over $\F_{1801}$, the inverse of the zeta function is given by

\begin{multline*}
\zeta^{-1}(\widetilde{X}_{\overline{\F}_{1801}}, t) = (1-t) (1- p^2 t) (1 - pt)^{14} \\
{}\times\bigl(1 - 1873\,t + 4068\,p\,t^2 - 5981\,p^2 t^3 + 4595\,p^3 t^4 \\
 - 5981\,p^4 t^5 + 4068\,p^5 t^6 - 1873\,p^6 t^7 + p^8 t^8\bigr).
\end{multline*}

\noindent Over $\F_{2251}$:

\begin{multline*}
\zeta^{-1}(\widetilde{X}_{\overline{\F}_{2251}}, t) = (1-t) (1- p^2 t) (1 - pt)^{10}(1 + pt)^4 \\
{}\times\bigl(1 + 2323\,t - 732\,p\,t^2 + 1181\,p^2 t^3 + 4595\,p^3 t^4 \\
 + 1181\,p^4 t^5 - 732\,p^5 t^6 + 2323\,p^6 t^7 + p^8 t^8\bigr).
\end{multline*}

\noindent Over $\F_{2251^2}$, with $q = 2251^2$:

\begin{multline*}
\zeta^{-1}(\widetilde{X}_{\overline{\F}_{2251^2}}, t) = (1-t) (1 - qt)^{14} \\
{}\times\bigl(1 - 8691793\,t + 15735588\,q\,t^2 - 16904231\,q^2 t^3 + 18737495\,q^3 t^4 \\
 - 16904231\,q^4 t^5 + 15735588\,q^5 t^6 - 8691793\,q^6 t^7 + q^8 t^8\bigr) (1-q^2 t) .
\end{multline*}

The orbifold cohomology has $24$ basis elements with $h^{0,0}=h^{0,2}=h^{2,0}=h^{2,2}=1$ and $h^{1,1}=20$. Among the $20$ elements in $H^{1,1}$, for $\F_{1801}$ all $14$ rational eigenvalues equal~$p$, while for $\F_{2251}$, $10$ have eigenvalue~$p$ and $4$ have eigenvalue~$-p$; in both cases the remaining $6$ lie in $\Q_p\setminus\Q$. The eigenvalues from $H^{0,2}$ and $H^{2,0}$ are also not rational. Substituting into~\eqref{eq:zeta-factored} yields a zeta function matching the one derived from \cite[Theorem~4.3.4]{GotoThesis}, verifying Conjecture~\ref{conj:supertrace-counts-points} over $\F_{1801}$ and $\F_{2251}$.

Over $\F_{q}$ for $q = {2251}^2$, each eigenvalue $\alpha_i$ is replaced by $\alpha_i^2$. The $4$ eigenvalues $-p$ appearing over $\F_{p}$ satisfy $(-p)^2 = p^2 = q$, so they now contribute $(1 - qt)$ factors to the zeta function denominator. The Berglund--H\"ubsch zeta function again agrees with Goto's calculation, verifying Conjecture~\ref{conj:supertrace-counts-points} in this case.
\end{example}

\begin{example}\label{ex:deformed-diagonal-16}
Consider the deformed diagonal surface
\[
x_0^{10} + x_1^{10} + x_2^2 + x_0 x_3^3 = 0 \quad\text{in}\;\mathbb P^3(1,1,5,3),
\]
with
\[
A = \begin{pmatrix} 10 & 0 & 0 & 0 \\ 0 & 10 & 0 & 0 \\ 0 & 0 & 2 & 0 \\ 1 & 0 & 0 & 3 \end{pmatrix},
\]
$\det(A) = 600$, $M = 30$, and $G = \langle J\rangle$. This surface has an isolated singularity of type~$A_{3,2}$ at $[0{:}0{:}0{:}1]$ (with $q_3 = 3$, $\alpha_3 = 2$, $r_3 = 2$), contributing a factor of $(1 - pt)^2$ to~$P_2$.

Calculating the terms from \cite[Theorem~4.3.4]{GotoThesis}, we see that the remaining contributions consist of a $(1-pt)$ factor, one $\mathfrak{V}$-term with eigenvalue~$p$, and $18$ $\mathfrak{W}$-terms. Two of the $\mathfrak{W}$-terms have Jacobi sum~$p$, giving a total of $6$ $(1 - pt)$ factors in the denominator of the zeta function. The other $16$ $\mathfrak{W}$-terms have complex Jacobi sums. At $p = 601$, the middle cohomology polynomial factors as

\begin{multline*}
P_2(t) = (1 - pt)^6 (1 - pt + p^2 t^2)^4 \\
{}\times \bigl(1 + 1532\,t + 2698\,p\,t^2 + 3704\,p^2 t^3 + 3955\,p^3 t^4 \\
 + 3704\,p^4 t^5 + 2698\,p^5 t^6 + 1532\,p^6 t^7 + p^8 t^8\bigr).
\end{multline*}

The orbifold cohomology has $24$ basis elements with $h^{0,0}=h^{0,2}=h^{2,0}=h^{2,2}=1$ and $h^{1,1}=20$. Of the $20$ elements in $H^{1,1}$, $6$ have eigenvalue~$p$ and $14$ lie in $\Q_p\setminus\Q$. Together with the elements from $H^{0,2}$ and $H^{2,0}$, these $16$ non-rational eigenvalues produce the $(1 - pt + p^2 t^2)^4$ factor and the degree~$8$ polynomial. Substituting into~\eqref{eq:zeta-factored} yields a zeta function matching the one derived from \cite[Theorem~4.3.4]{GotoThesis}.
\end{example}

In the next example, we will examine a K3 surface that is neither diagonal, nor deformed-diagonal.

\begin{example}\label{ex:L2L2}
Consider the K3 surface defined by the $L_2 L_2$ potential
\[
W_A = x_0^3 x_1 + x_0 x_1^3 + x_2^3 x_3 + x_2 x_3^3 = 0 \quad\text{in}\;\mathbb P^3(1,1,1,1),
\]
where $A$ is the block diagonal matrix
\[
A = \begin{pmatrix} 3 & 1 & 0 & 0 \\ 1 & 3 & 0 & 0 \\ 0 & 0 & 3 & 1 \\ 0 & 0 & 1 & 3 \end{pmatrix}
\]
composed of two $L_2(3,3)$ loop blocks. Here $d=4$, $\det(A)=64$, and $G=\langle J\rangle$ has order~$4$.

 The orbifold cohomology has $24$ basis elements with the same Hodge diamond as the previous example: $h^{0,0}=h^{0,2}=h^{2,0}=h^{2,2}=1$ and $h^{1,1}=20$. All $20$ elements in $H^{1,1}$ have eigenvalue~$p$: each twisted element's $\Gamma_p$ product reduces to $\pm 1$ via complementary entries $z, 1 - z \in \gamma A^{-1}$ contributing $\Gamma_p(1 - z)\Gamma_p (z)= \pm 1$ to the eigenvalue.  The only nontrivial eigenvalues come from $H^{0,2}$ and $H^{2,0}$, which contribute a single irreducible factor of degree~$2$.

Table~\ref{tab:L2L2-point-counts} lists the point counts $N_p$ at the first five primes $p$ satisfying $\det(A) = 64 \mid p-1$, where $A$ is Berglund--H\"ubsch in the sense of Definition~\ref{def:berglund-hubsch}. In each case the orbifold supertrace and the brute-force projective enumeration agree; both are computed in \cite{P}.

\begin{table}[ht]
\centering
\caption{Point counts for the $L_2 L_2$ surface at primes with $64\mid p-1$.}\label{tab:L2L2-point-counts}
\begin{tabular}{@{}rr@{}}
\toprule
$p$ & $N_p$ \\
\midrule
$193$ & $40{,}920$ \\
$257$ & $70{,}680$ \\
$449$ & $209{,}880$ \\
$577$ & $343{,}320$ \\
$641$ & $424{,}920$ \\
\bottomrule
\end{tabular}
\end{table}

At $p = 193$, the smallest prime with $\det(A) = 64 \mid p-1$, the orbifold zeta function is
\[
\zeta_p^{-1}(A, G, t) = (1-t) (1-pt)^{20}\bigl(1 + 190\,t + p^2 t^2\bigr) (1 - p^2t)\,.
\]
We list all $24$ orbifold trace contributions in Table~\ref{tab:L2L2-traces}.
All $20$ eigenvalues in $H^{1,1}$ equal~$p$; the two eigenvalues from $H^{0,2}$ and $H^{2,0}$ contribute to the quadratic factor $1+190\,t+p^2 t^2$, with $193$-adic expansions starting with $\alpha_{0,2} = 3 + 192\cdot 193 + 63\cdot 193^2 + 21\cdot 193^3 + \cdots$ and $\alpha_{2,0} = 129\cdot 193^2 + 171\cdot 193^3 + 135\cdot 193^4 + 47\cdot 193^5 + \cdots$

As it happens, the conjecture also holds at the prime $p = 281$, even though $\det(A) \nmid 281 - 1$. In this case, we compute the zeta function as
\[
\zeta_p^{-1}(A, G, t) = (1-t) (1-pt)^{20}\bigl(1 + 462\,t + p^2 t^2\bigr) (1 - p^2t)\,,
\]
which matches the zeta function shown in \cite[Table~3]{Whitcher21} for $\psi = 0$.

\begin{table}[ht]
\centering
\caption{Orbifold trace contributions for the $L_2 L_2$ surface at $p=193$.}\label{tab:L2L2-traces}
\renewcommand{\arraystretch}{1.1}
{\footnotesize
\begin{tabular}{@{}llll@{}}
\toprule
$H^{s,r}$ & Element & $\gamma A^{-1}$ & $\alpha_i$ \\
\midrule
$H^{0,0}$ & $y_1 y_2 y_3 y_4$ & $(0,0,0,0)$ & $1$ \\
\midrule
$H^{0,2}$ & $x_1 x_2 x_3 x_4\, e_1 e_2 e_3 e_4$ & $(\tfrac14,\tfrac14,\tfrac14,\tfrac14)$ & $\alpha_{0,2}$ \\
\midrule
$H^{1,1}$ & $y_1^2 y_2^2 y_3^2 y_4^2$ & $(0,0,0,0)$ & $p$ \\
 & $x_1 x_2 x_3^3 x_4^3\, e_1 e_2 e_3 e_4$ & $(\tfrac14,\tfrac14,\tfrac34,\tfrac34)$ & $p$ \\
 & $x_1 x_2^2 x_3^2 x_4^3\, e_1 e_2 e_3 e_4$ & $(\tfrac18,\tfrac58,\tfrac38,\tfrac78)$ & $p$ \\
 & $x_1 x_2^2 x_3^3 x_4^2\, e_1 e_2 e_3 e_4$ & $(\tfrac18,\tfrac58,\tfrac78,\tfrac38)$ & $p$ \\
 & $x_1 x_2^3 x_3 x_4^3\, e_1 e_2 e_3 e_4$ & $(0,1,0,1)$ & $p$ \\
 & $x_1 x_2^3 x_3^2 x_4^2\, e_1 e_2 e_3 e_4$ & $(0,1,\tfrac12,\tfrac12)$ & $p$ \\
 & $x_1 x_2^3 x_3^3 x_4\, e_1 e_2 e_3 e_4$ & $(0,1,1,0)$ & $p$ \\
 & $x_1^2 x_2 x_3^2 x_4^3\, e_1 e_2 e_3 e_4$ & $(\tfrac58,\tfrac18,\tfrac38,\tfrac78)$ & $p$ \\
 & $x_1^2 x_2 x_3^3 x_4^2\, e_1 e_2 e_3 e_4$ & $(\tfrac58,\tfrac18,\tfrac78,\tfrac38)$ & $p$ \\
 & $x_1^2 x_2^2 x_3 x_4^3\, e_1 e_2 e_3 e_4$ & $(\tfrac12,\tfrac12,0,1)$ & $p$ \\
 & $x_1^2 x_2^2 x_3^2 x_4^2\, e_1 e_2 e_3 e_4$ & $(\tfrac12,\tfrac12,\tfrac12,\tfrac12)$ & $p$ \\
 & $x_1^2 x_2^2 x_3^3 x_4\, e_1 e_2 e_3 e_4$ & $(\tfrac12,\tfrac12,1,0)$ & $p$ \\
 & $x_1^2 x_2^3 x_3 x_4^2\, e_1 e_2 e_3 e_4$ & $(\tfrac38,\tfrac78,\tfrac18,\tfrac58)$ & $p$ \\
 & $x_1^2 x_2^3 x_3^2 x_4\, e_1 e_2 e_3 e_4$ & $(\tfrac38,\tfrac78,\tfrac58,\tfrac18)$ & $p$ \\
 & $x_1^3 x_2 x_3 x_4^3\, e_1 e_2 e_3 e_4$ & $(1,0,0,1)$ & $p$ \\
 & $x_1^3 x_2 x_3^2 x_4^2\, e_1 e_2 e_3 e_4$ & $(1,0,\tfrac12,\tfrac12)$ & $p$ \\
 & $x_1^3 x_2 x_3^3 x_4\, e_1 e_2 e_3 e_4$ & $(1,0,1,0)$ & $p$ \\
 & $x_1^3 x_2^2 x_3 x_4^2\, e_1 e_2 e_3 e_4$ & $(\tfrac78,\tfrac38,\tfrac18,\tfrac58)$ & $p$ \\
 & $x_1^3 x_2^2 x_3^2 x_4\, e_1 e_2 e_3 e_4$ & $(\tfrac78,\tfrac38,\tfrac58,\tfrac18)$ & $p$ \\
 & $x_1^3 x_2^3 x_3 x_4\, e_1 e_2 e_3 e_4$ & $(\tfrac34,\tfrac34,\tfrac14,\tfrac14)$ & $p$ \\
\midrule
$H^{2,0}$ & $x_1^3 x_2^3 x_3^3 x_4^3\, e_1 e_2 e_3 e_4$ & $(\tfrac34,\tfrac34,\tfrac34,\tfrac34)$ & $\alpha_{2,0}$ \\
\midrule
$H^{2,2}$ & $y_1^3 y_2^3 y_3^3 y_4^3$ & $(0,0,0,0)$ & $p^2$ \\
\bottomrule
\end{tabular}
}
\renewcommand{\arraystretch}{1.0}
\end{table}
\end{example}

\begin{example}\label{ex:chain-3334}
Let
\[
A = \mat{3 & 1 & 0 & 0 \\ 0 & 3 & 1 & 0 \\ 0 & 0 & 3 & 1 \\ 0 & 0 & 0 & 4}
\]
and $G=\langle J\rangle$. In particular, $X_A$ is a smooth (non-diagonal, non-deformed-diagonal) K3 surface in $\mathbb{P}^3$. The 24 generators of the orbifold cohomology and their contributions to $ST_p(A,\langle J\rangle)$ are shown in  Table~\ref{tab:chain-3334-traces}.

\begin{table}[ht]
\centering
\caption{Orbifold trace contributions for Chain$(3,3,3,4)$.}\label{tab:chain-3334-traces}
\renewcommand{\arraystretch}{1.1}
{\scriptsize
\begin{tabular}{@{}llll@{}}
\toprule
$H^{s,r}$ & Element & $\gamma A^{-1}$ & $\alpha_i$ \\
\midrule
$H^{0,0}$ & $y_1 y_2 y_3 y_4$ & $(0,0,0,0)$ & $1$ \\
\midrule
$H^{0,2}$ & $x_1 x_2 x_3 x_4\, e_1 e_2 e_3 e_4$ & $(\tfrac{1}{3},\tfrac{2}{9},\tfrac{7}{27},\tfrac{5}{27})$ & $-\,\Gamma_p(\tfrac{1}{3})\,\Gamma_p(\tfrac{2}{9})\,\Gamma_p(\tfrac{7}{27})\,\Gamma_p(\tfrac{5}{27})$ \\
\midrule
$H^{1,1}$ & $y_1^2 y_2^2 y_3^2 y_4^2$ & $(0,0,0,0)$ & $p$ \\
 & $x_1 x_2 x_3^2 x_4^4\, e_1 e_2 e_3 e_4$ & $(\tfrac{1}{3},\tfrac{2}{9},\tfrac{16}{27},\tfrac{23}{27})$ & $p\,\Gamma_p(\tfrac{1}{3})\,\Gamma_p(\tfrac{2}{9})\,\Gamma_p(\tfrac{16}{27})\,\Gamma_p(\tfrac{23}{27})$ \\
 & $x_1 x_2 x_3^3 x_4^3\, e_1 e_2 e_3 e_4$ & $(\tfrac{1}{3},\tfrac{2}{9},\tfrac{25}{27},\tfrac{14}{27})$ & $p\,\Gamma_p(\tfrac{1}{3})\,\Gamma_p(\tfrac{2}{9})\,\Gamma_p(\tfrac{25}{27})\,\Gamma_p(\tfrac{14}{27})$ \\
 & $x_1 x_2^2 x_3 x_4^4\, e_1 e_2 e_3 e_4$ & $(\tfrac{1}{3},\tfrac{5}{9},\tfrac{4}{27},\tfrac{26}{27})$ & $p\,\Gamma_p(\tfrac{1}{3})\,\Gamma_p(\tfrac{5}{9})\,\Gamma_p(\tfrac{4}{27})\,\Gamma_p(\tfrac{26}{27})$ \\
 & $x_1 x_2^2 x_3^2 x_4^3\, e_1 e_2 e_3 e_4$ & $(\tfrac{1}{3},\tfrac{5}{9},\tfrac{13}{27},\tfrac{17}{27})$ & $p\,\Gamma_p(\tfrac{1}{3})\,\Gamma_p(\tfrac{5}{9})\,\Gamma_p(\tfrac{13}{27})\,\Gamma_p(\tfrac{17}{27})$ \\
 & $x_1 x_2^2 x_3^3 x_4^2\, e_1 e_2 e_3 e_4$ & $(\tfrac{1}{3},\tfrac{5}{9},\tfrac{22}{27},\tfrac{8}{27})$ & $p\,\Gamma_p(\tfrac{1}{3})\,\Gamma_p(\tfrac{5}{9})\,\Gamma_p(\tfrac{22}{27})\,\Gamma_p(\tfrac{8}{27})$ \\
 & $x_1 x_2^3 x_3 x_4^3\, e_1 e_2 e_3 e_4$ & $(\tfrac{1}{3},\tfrac{8}{9},\tfrac{1}{27},\tfrac{20}{27})$ & $p\,\Gamma_p(\tfrac{1}{3})\,\Gamma_p(\tfrac{8}{9})\,\Gamma_p(\tfrac{1}{27})\,\Gamma_p(\tfrac{20}{27})$ \\
 & $x_1 x_2^3 x_3^2 x_4^2\, e_1 e_2 e_3 e_4$ & $(\tfrac{1}{3},\tfrac{8}{9},\tfrac{10}{27},\tfrac{11}{27})$ & $p\,\Gamma_p(\tfrac{1}{3})\,\Gamma_p(\tfrac{8}{9})\,\Gamma_p(\tfrac{10}{27})\,\Gamma_p(\tfrac{11}{27})$ \\
 & $x_1 x_2^3 x_3^3 x_4\, e_1 e_2 e_3 e_4$ & $(\tfrac{1}{3},\tfrac{8}{9},\tfrac{19}{27},\tfrac{2}{27})$ & $p\,\Gamma_p(\tfrac{1}{3})\,\Gamma_p(\tfrac{8}{9})\,\Gamma_p(\tfrac{19}{27})\,\Gamma_p(\tfrac{2}{27})$ \\
 & $x_1^2 x_2 x_3 x_4^4\, e_1 e_2 e_3 e_4$ & $(\tfrac{2}{3},\tfrac{1}{9},\tfrac{8}{27},\tfrac{25}{27})$ & $p\,\Gamma_p(\tfrac{2}{3})\,\Gamma_p(\tfrac{1}{9})\,\Gamma_p(\tfrac{8}{27})\,\Gamma_p(\tfrac{25}{27})$ \\
 & $x_1^2 x_2 x_3^2 x_4^3\, e_1 e_2 e_3 e_4$ & $(\tfrac{2}{3},\tfrac{1}{9},\tfrac{17}{27},\tfrac{16}{27})$ & $p\,\Gamma_p(\tfrac{2}{3})\,\Gamma_p(\tfrac{1}{9})\,\Gamma_p(\tfrac{17}{27})\,\Gamma_p(\tfrac{16}{27})$ \\
 & $x_1^2 x_2 x_3^3 x_4^2\, e_1 e_2 e_3 e_4$ & $(\tfrac{2}{3},\tfrac{1}{9},\tfrac{26}{27},\tfrac{7}{27})$ & $p\,\Gamma_p(\tfrac{2}{3})\,\Gamma_p(\tfrac{1}{9})\,\Gamma_p(\tfrac{26}{27})\,\Gamma_p(\tfrac{7}{27})$ \\
 & $x_1^2 x_2^2 x_3 x_4^3\, e_1 e_2 e_3 e_4$ & $(\tfrac{2}{3},\tfrac{4}{9},\tfrac{5}{27},\tfrac{19}{27})$ & $p\,\Gamma_p(\tfrac{2}{3})\,\Gamma_p(\tfrac{4}{9})\,\Gamma_p(\tfrac{5}{27})\,\Gamma_p(\tfrac{19}{27})$ \\
 & $x_1^2 x_2^2 x_3^2 x_4^2\, e_1 e_2 e_3 e_4$ & $(\tfrac{2}{3},\tfrac{4}{9},\tfrac{14}{27},\tfrac{10}{27})$ & $p\,\Gamma_p(\tfrac{2}{3})\,\Gamma_p(\tfrac{4}{9})\,\Gamma_p(\tfrac{14}{27})\,\Gamma_p(\tfrac{10}{27})$ \\
 & $x_1^2 x_2^2 x_3^3 x_4\, e_1 e_2 e_3 e_4$ & $(\tfrac{2}{3},\tfrac{4}{9},\tfrac{23}{27},\tfrac{1}{27})$ & $p\,\Gamma_p(\tfrac{2}{3})\,\Gamma_p(\tfrac{4}{9})\,\Gamma_p(\tfrac{23}{27})\,\Gamma_p(\tfrac{1}{27})$ \\
 & $x_1^2 x_2^3 x_3 x_4^2\, e_1 e_2 e_3 e_4$ & $(\tfrac{2}{3},\tfrac{7}{9},\tfrac{2}{27},\tfrac{13}{27})$ & $p\,\Gamma_p(\tfrac{2}{3})\,\Gamma_p(\tfrac{7}{9})\,\Gamma_p(\tfrac{2}{27})\,\Gamma_p(\tfrac{13}{27})$ \\
 & $x_1^2 x_2^3 x_3^2 x_4\, e_1 e_2 e_3 e_4$ & $(\tfrac{2}{3},\tfrac{7}{9},\tfrac{11}{27},\tfrac{4}{27})$ & $p\,\Gamma_p(\tfrac{2}{3})\,\Gamma_p(\tfrac{7}{9})\,\Gamma_p(\tfrac{11}{27})\,\Gamma_p(\tfrac{4}{27})$ \\
 & $x_1^3 x_2 x_3 x_4^3\, e_1 e_2 e_3 e_4$ & $(1,0,\tfrac{1}{3},\tfrac{2}{3})$ & $p\,\Gamma_p(1)\,\Gamma_p(\tfrac{1}{3})\,\Gamma_p(\tfrac{2}{3})$ \\
 & $x_1^3 x_2 x_3^2 x_4^2\, e_1 e_2 e_3 e_4$ & $(1,0,\tfrac{2}{3},\tfrac{1}{3})$ & $p\,\Gamma_p(1)\,\Gamma_p(\tfrac{2}{3})\,\Gamma_p(\tfrac{1}{3})$ \\
 & $x_1^3 x_2 x_3^3 x_4\, e_1 e_2 e_3 e_4$ & $(1,0,1,0)$ & $p\,\Gamma_p(1)^2$ \\
\midrule
$H^{2,0}$ & $x_1^2 x_2^3 x_3^3 x_4^4\, e_1 e_2 e_3 e_4$ & $(\tfrac{2}{3},\tfrac{7}{9},\tfrac{20}{27},\tfrac{22}{27})$ & $-p^2\,\Gamma_p(\tfrac{2}{3})\,\Gamma_p(\tfrac{7}{9})\,\Gamma_p(\tfrac{20}{27})\,\Gamma_p(\tfrac{22}{27})$ \\
\midrule
$H^{2,2}$ & $y_1^3 y_2^3 y_3^3 y_4^3$ & $(0,0,0,0)$ & $p^2$ \\
\bottomrule
\end{tabular}
}
\renewcommand{\arraystretch}{1.0}
\end{table}

If $p=5$, then the degree of $X_A$ divides $p-1$ but $\det(A)=108\nmid 4$. Similarly to Example \ref{ex:3.3}, we have
\[
ST_5(A,\langle J\rangle)=4 + 2\cdot 5 + 4\cdot 5^2 + 4\cdot 5^3 + 5^4 + 5^5 + 5^7 + 2\cdot 5^9 + \cdots \in\Z_5,
\]
which is a non-integer rational number while a brute-force calculation shows that $\#X_A(\mathbb F_5)=36$. On the other hand, if $p=109$ (the smallest prime such that $p-1$ is divisible by $\det(A)=108$ then $ST_{109}(A,\langle J\rangle)=12{,}147$, matches the brute-force count.
\end{example}

\section{The Fermat Quintic and Its Quotients}\label{sec:quintic-and-quotients}

We illustrate the orbifold trace zeta calculation with the Fermat quintic threefold, where the zeta function has been computed by Candelas, de la~Ossa, and Rodriguez~Villegas~\cite{CdlORV03}.

\begin{example}\label{ex:quintic}
Consider the Fermat quintic threefold
\[
M : x_0^5 + x_1^5 + x_2^5 + x_3^5 + x_4^5 = 0 \quad\text{in}\;\mathbb P^4,
\]
with $A=5I_5$, degree $d=5$, $Q=(1,1,1,1,1)$, $\det(A)=5^5$, and $G=\langle J\rangle$ of order~$5$. In \cite{CdlORV03}, the zeta function for the quintic is described for all possible values of $\rho$, the multiplicative order of~$p$ modulo~$5$. By Euler's theorem, the possible values are $\rho\in\{1,2,4\}$. \cite[Equation~11.12]{CdlORV03} gives the zeta function of the one-parameter family $M_\psi$ as
\[
\zeta_p(M_\psi, t) = \frac{R_1(t,\psi)\,R_A(p^\rho t^\rho,\psi)^{20/\rho}\,R_B(p^\rho t^\rho,\psi)^{30/\rho}}{(1-t)(1-pt)(1-p^2 t)(1-p^3 t)},
\]
where $R_1$, $R_A$, $R_B$ are quartic polynomials depending on $p$ and~$\psi$. The notation $R_A^2$ refers to the role of $R_A$ in the Euler curve's zeta function $\zeta_A=R_A^2/[(1-u)(1-pu)]$; the polynomial $R_A(pt)^2$ itself is quartic. At the Fermat point $\psi=0$ with $\rho=1$, one has $R_A=R_B$ and the numerator becomes
\[
P_3(t) = R_1(t)\cdot\bigl[R_A(pt)^2\bigr]^{50}
\]
with $\deg P_3 = 4 + 4\times 50 = 204 = b_3(M)$. Following the parametrization of \cite[Table~12.1]{CdlORV03}, the quartics take the form
\begin{align}
R_1(t) &= 1+a_1\,t + b_1\,p\,t^2 + a_1\,p^3\,t^3 + p^6\,t^4,\label{eq:6.1}\\
R_A(pt)^2 &= 1+c\,p\,t+d\,p^2\,t^2+c\,p^4\,t^3+p^6\,t^4,\label{eq:6.2}
\end{align}
where the integer coefficients $(a_1,b_1)$ and $(c,d)$ depend on~$p$.

The orbifold cohomology has $208$ basis elements with Hodge numbers $h^{0,0}=h^{3,3}=1$, $h^{1,1}=h^{2,2}=1$, $h^{0,3}=h^{3,0}=1$, $h^{1,2}=h^{2,1}=101$. All elements in $H^{1,2}$ have ${\rm age}^\vee(\gamma) = 2$ and $\gamma A^{-1}=(a_1/5,\dots,a_5/5)$ with each $a_i\in\{1,2,3,4\}$ and $\sum a_i = 10$. Elements in $H^{2,1}$ are of the same form with ${\rm age}^\vee(\gamma)=3$ and $\sum a_i=15$. The $204$ elements with odd total Hodge degree $s+r=3$---one from each of $H^{0,3}$ and $H^{3,0}$, and $101$ from each of $H^{1,2}$ and $H^{2,1}$---contribute to the numerator~$P_3(t)$, while the remaining~$4$ with even $s+r$ give the denominator.

By~\eqref{eq:zeta-factored}, the Berglund--H\"ubsch zeta function factors as $\zeta_p(A,\langle J\rangle,t) = \prod_k P_k(t)^{(-1)^{k+1}}$, where $P_k(t)=\prod_{s_i+r_i=k}(1-\alpha_i\,t)$. The even-degree factors are
\begin{alignat*}{2}
P_0(t) &= 1-t, &\qquad P_4(t) &= 1-p^2\,t,\\
P_2(t) &= 1-p\,t, & P_6(t) &= 1-p^3\,t,
\end{alignat*}
and the full zeta function should assemble as
\[
\zeta_p(A,\langle J\rangle,t)
= \frac{P_3(t)}{(1-t)(1-pt)(1-p^2 t)(1-p^3 t)}
= \frac{R_1(t)\cdot\bigl[R_A(pt)^2\bigr]^{50}}{(1-t)(1-pt)(1-p^2 t)(1-p^3 t)}\,.
\]

\medskip
We now verify the conjecture at the first four primes satisfying $\det(A) = 3125\mid (p-1)$ and $\rho = 1$, namely $p=37501$, $62501$, $112501$, and~$118751$. At these sizes, brute-force enumeration is infeasible. Instead, we compare the Berglund--H\"ubsch point count against the point count given by \cite[Equation~9.19]{CdlORV00}. Moreover, we compare the orbifold zeta function against the zeta function calculated as in \cite{CdlORV00}, which computes $R_1$ and $R_A(pt)^2$ by solving for the coefficients $a_1, b_1, c$ and $d$ from explicitly calculated point counts. The two methods produce identical degree-$204$ numerator polynomials $P_3(t)$ in each case; the quartic coefficients and point counts are listed in Table~\ref{tab:quintic-large-primes}. All computations are carried out in~\cite{P}.

\begin{table}[ht]
\centering
\caption{Quartic factors and point counts at the first four primes with $\det(A)=3125\mid p-1$.}\label{tab:quintic-large-primes}
\begin{tabular}{@{}rrrrrr@{}}
\toprule
$p$ & $a_1$ & $b_1$ & $c$ & $d$ & $N_p$ \\
\midrule
$37501$  & $-8414879$  & $1287051631$    & $271$  & $93331$  & $52740499948675$ \\
$62501$  & $30690371$  & $9257997381$    & $71$   & $99981$  & $244156502943925$ \\
$112501$ & $17212621$  & $-915109619$    & $-479$ & $278581$ & $1423876073488675$ \\
$118751$ & $11436371$  & $-14628025869$  & $521$  & $275331$ & $1674620058737425$ \\
\bottomrule
\end{tabular}
\end{table}

\noindent Substituting the coefficients from Table~\ref{tab:quintic-large-primes} into \eqref{eq:6.1} and \eqref{eq:6.2} or factoring the Berglund--H\"ubsch zeta function directly gives the following explicit quartic factors.

At $p=37501$:
\begin{align*}
R_1(t) &= 1-8414879\,t+1287051631\,p\,t^2-8414879\,p^3\,t^3+p^6\,t^4,\\
R_A(pt)^2 &= 1+271\,p\,t+93331\,p^2\,t^2+271\,p^4\,t^3+p^6\,t^4.
\end{align*}

At $p=62501$:
\begin{align*}
R_1(t) &= 1+30690371\,t+9257997381\,p\,t^2+30690371\,p^3\,t^3+p^6\,t^4,\\
R_A(pt)^2 &= 1+71\,p\,t+99981\,p^2\,t^2+71\,p^4\,t^3+p^6\,t^4.
\end{align*}

At $p=112501$:
\begin{align*}
R_1(t) &= 1+17212621\,t-915109619\,p\,t^2+17212621\,p^3\,t^3+p^6\,t^4,\\
R_A(pt)^2 &= 1-479\,p\,t+278581\,p^2\,t^2-479\,p^4\,t^3+p^6\,t^4.
\end{align*}

At $p=118751$:
\begin{align*}
R_1(t) &= 1+11436371\,t-14628025869\,p\,t^2+11436371\,p^3\,t^3+p^6\,t^4,\\
R_A(pt)^2 &= 1+521\,p\,t+275331\,p^2\,t^2+521\,p^4\,t^3+p^6\,t^4.
\end{align*}

\noindent We list the point counts $N_{p^\nu}=\#M(\F_{p^\nu})$ extracted from the power series expansion $\log Z(t)=\sum_{\nu\geq 1}N_{p^\nu}\,t^\nu/\nu$ for $\nu=1,2,3$ in Table~\ref{tab:quintic-point-counts}.

\begin{table}[ht]
\centering
\caption{Point counts over $\F_{p^\nu}$ for the Fermat quintic at the first four primes $p$ with $3125\mid (p-1)$.}\label{tab:quintic-point-counts}
\renewcommand{\arraystretch}{1.1}
\adjustbox{max width=\textwidth}{%
\footnotesize
\begin{tabular}{@{}rlll@{}}
\toprule
$p$ & $N_p$ & $N_{p^2}$ & $N_{p^3}$ \\
\midrule
$37501$ & $52740499948675$ & $2781359284579565153342704675$ & $146684977590415830796619713088162291370925$ \\
$62501$ & $244156502943925$ & $59610367065473834056333248175$ & $14554010838275253898527781005005357802359425$ \\
$112501$ & $1423876073488675$ & $2027394654052389497109812802175$ & $2886738507011079685634568411080155451821512175$ \\
$118751$ & $1674620058737425$ & $2804294711853468324003533172175$ & $4696079921757156471284481243293773424762344675$ \\
\bottomrule
\end{tabular}
}
\end{table}

As it happens, the conjecture also holds for the $\rho = 1$ primes listed in \cite{CdlORV03}: $11$, $31$, $41$, $61$, $71$, $101$.
In every case the orbifold trace formula reproduces the calculations of \cite{CdlORV03} exactly.
\end{example}

\begin{example}\label{ex:greene-plesser}
Let $A$ be the diagonal matrix $5I_5$ and let $X_A$ denote the Fermat quintic $x_1^5 + \cdots + x_5^5 = 0$ in $\mathbb{P}^4$. Greene and Plesser \cite{GP} showed that quotients $[X_A(\mathbb C)/G]$ by diagonal phase symmetry groups $G$ yield a pair of Calabi--Yau 3-folds exhibiting mirror symmetry at the level of Hodge numbers. We recover this from the orbifold trace formula for one such mirror pair with $\chi = \pm 8$.

\medskip\noindent\textbf{The $|G|=125$ orbifold.}
Taking the symmetry group $G = \langle J,\,(0,1,2,3,4),\,(0,1,1,4,4)\rangle$ of order~$125$, the orbifold $[X_A/G]$ has $80$ cohomology basis elements with Chen--Ruan Hodge numbers $h^{1,1}=17$, $h^{2,1}=21$. The orbifold zeta function calculated using \cite{P} is
\begin{equation}\label{eq:6.3}
\zeta_p(A, G,t)
= \frac{R_1(t)\cdot\bigl[R_A(pt)^2\bigr]^{10}}{(1-t)(1-pt)^{17}(1-p^2 t)^{17}(1-p^3 t)},
\end{equation}
where $R_1$ and $R_A(pt)^2$ are, respectively, as in \eqref{eq:6.1} and \eqref{eq:6.2} with $(a_1,b_1)=(-89,351)$ and $(c,d)=(1,-9)$. In accordance with Theorem \ref{thm:3.6}, all $44$ roots of the numerator satisfy the Riemann hypothesis $|\alpha_i|=p^{3/2}$ and the Chen--Ruan Euler characteristic of this orbifold, $-8$ matches the difference between the degree of the denominator and that of the numerator.

\medskip\noindent\textbf{The $|G^T|=25$ mirror orbifold.}
The dual group $G^T = \langle J,\,(0,4,1,1,4)\rangle$ has order~$25$. The orbifold $[X_A/G^T]$ again has $80$ basis elements but with mirror-exchanged Hodge numbers $h^{1,1}=21$, $h^{2,1}=17$. The orbifold zeta function calculated using \cite{P} is
\begin{equation}\label{eq:6.4}
\zeta_p(A, G^T,t)
= \frac{R_1(t)\cdot\bigl[R_A(pt)^2\bigr]^{8}}{(1-t)(1-pt)^{21}(1-p^2 t)^{21}(1-p^3 t)},
\end{equation}
with the same $R_1$ and $R_A(pt)^2$ as in \eqref{eq:6.3}. In accordance with Theorem \ref{thm:3.6}, all $36$ roots of the numerator satisfy the Riemann hypothesis $|\alpha_i|=p^{3/2}$ and the Chen--Ruan Euler characteristic of this orbifold, $8$ matches the difference between the degree of the denominator and that of the numerator.
\end{example}

\section{Alternative Proof via Monsky--Washnitzer Cohomology}\label{sec:alternative-proof}

In this section we prove Conjecture~\ref{conj:supertrace-counts-points} for Fermat-type Berglund--H\"ubsch hypersurfaces by comparing the orbifold supertrace against the point-counting formula from Monsky--Washnitzer cohomology~\cite{M}. The idea is to decompose the point count into terms matching the orbifold eigenvalues, and show that the remaining terms cancel pairwise. We then illustrate how this cancellation extends with some modifications to an explicit non-diagonal example. Throughout this section, we restrict our attention exclusively to Berglund--H\"ubsch $n \times n$ matrices $A$ such that the potential $W_A$ is a homogeneous Calabi--Yau polynomial. This ensures that $X_A \subset \mathbb{P}^{n-1}$ is a standard unweighted projective hypersurface, allowing for direct application of Monsky's point-counting formula. We further assume $n\mid(p-1)$.

By \cite{M}, the number of $\F_p$-rational points on a projective hypersurface can be expressed in terms of a Frobenius operator $\Fr$ acting on Monsky--Washnitzer overconvergent cohomology. Concretely, let $\cR^\dagger\subset K[\![x_0, x_1,\ldots,x_n]\!]$ (where $K=\Q_p(\pi)$, $\pi^{p-1}=-p$) be the ring of overconvergent series $\sum_\gamma a_\gamma(\pi x_0)^{|\gamma A^{-1}|}x^\gamma$ with $|\gamma A^{-1}|\in\N$, and let $\cR_I^\dagger\subset\cR^\dagger$ be the subspace supported on monomials $x^\gamma$ with $\gamma_i>0$ for all $i\in I$. By \cite[Lemma~7.4]{M}, the number of $\F_p$-points on $X_A\colon\{W_A=0\}\subset\mathbb{P}^{n-1}$ is
\begin{equation}\label{eq:mw-point-count}
N_{MW}(X_A) = \frac{p^{n-1}-1}{p-1} + \frac{(-1)^n}{p}\sum_{I\subseteq\{1,\ldots,n\}} (-p)^{|\overline{I}|}\,\Tr\!\left(\left.\Fr\right|_{\cR_I^\dagger}\right).
\end{equation}
In this setting, the Frobenius operator is $\Fr = \Psi \circ \widetilde{Z}_A$, where $\widetilde{Z}_A$ is multiplication by
\[
e^{\pi(x_0 W_A(x)-x_0^p W_{pA}(x))}
= \sum_k c_k \pi^{|k|} x_0^{|k|} x^{kA}\,,
\]
and $\Psi$ maps $x_0^{\gamma_0}x^\gamma \mapsto x_0^{\gamma_0/p}x^{\gamma/p}$ if $p$ divides all exponents, and to $0$ otherwise. Representing this Frobenius operator with respect to the standard monomial basis, we find that the diagonal entry corresponding to $x_0^{|\gamma A^{-1}|}x^\gamma$ is $c_k \pi^{|k|}$ where $\gamma +kA=p\gamma$, which implies $k = (p-1)\gamma A^{-1}$. Note that since $k \in \N^n$, this relation can only be solved for $k$ if $(\gamma A^{-1})_i \geq 0$. Because $A$ is Berglund--H\"ubsch, the vector $k$ is integer-valued. Furthermore, $|k + \gamma A^{-1}| = p|\gamma A^{-1}|$ and the trace reduces to a sum of diagonal entries:

\begin{equation}\label{eq:trace-diagonal}
\Tr\!\left(\left.\Fr\right|_{\cR_I^\dagger}\right) = \sum_{\gamma} c_{(p-1)\gamma A^{-1}}\,(-p)^{|\gamma A^{-1}|},
\end{equation}
where the sum is over all $\gamma$ in the cone $\gamma A^{-1}\ge 0$ with $|\gamma A^{-1}|\in\N$ and $(\gamma A^{-1})_i>0$ for $i\in I$, and $c_k$ are the coefficients of the function $e^{\pi(W_A(x)-W_{pA}(x))}=\sum_k c_k \pi^{|k|}x^{kA}$.

Grouping summands of the form $\gamma + kA$ for fixed $\gamma$ (with $0\le(\gamma A^{-1})_i\le 1$) and variable $k\in\N^n$, we define
\begin{equation}\label{eq:s-gamma}
S(\gamma) = (-p)^{|\gamma A^{-1}|}\sum_{k\in\N^n} c_{(p-1)\gamma A^{-1}+(p-1)k}\,(-p)^{|k|}.
\end{equation}

\begin{lemma}\label{lemma:s-gamma-gauss}
Let $A$ be Berglund--H\"ubsch and suppose that $\gamma$ satisfies $0\le(\gamma A^{-1})_i\le 1$ for all $i$. Then,
\begin{equation}\label{eq:s-gamma-identity}
(p-1)^n\,S(\gamma) = (-1)^n\,(-p)^{|\gamma A^{-1}|}\,\Gamma_p(\gamma A^{-1}).
\end{equation}
\end{lemma}

\begin{proof}
For $\alpha \in \frac{1}{p-1} \Z \setminus \Z$ and $t  = (p-1)(1 - \alpha)$, the Gross--Koblitz formula (\cite[Theorem~4]{A}) implies that
\begin{equation}\label{eq:s_gamma_gamma_p_identity}
(p-1) \sum_{m \geq 0} c_{(p-1)m + (p - 1)\langle \alpha \rangle} (-p)^m \pi^{(p-1)\langle \alpha \rangle} = p (-p)^{\langle \alpha \rangle - 1} \Gamma_p\left( \langle \alpha \rangle \right)\,,
\end{equation}
where $c_m$ denotes the $1$-dimensional coefficients (i.e., for $n=1$), $\langle \alpha \rangle = 0$ for $\alpha = 0$, and is otherwise the unique rational number satisfying $0 < \langle \alpha \rangle \leq 1$  and $\alpha \equiv \langle \alpha \rangle \pmod{\Z}$.

In fact, equation \eqref{eq:s_gamma_gamma_p_identity} also holds for $\alpha = 1$.
To see that, note that the left hand side of the equation becomes
\[
(p-1)  \sum_{m \geq 1} c_{(p-1)m} (-p)^{m} = G_0 - (p - 1) = -p\,,
\]
where we used the fact that $G_0 = -1$, while the right hand side is $p \Gamma_p(1) = -p$, as claimed. A similar argument establishes \eqref{eq:s_gamma_gamma_p_identity} for $\alpha = 0$.

Using $\alpha = \left( \gamma A^{-1} \right)_i$ in \eqref{eq:s_gamma_gamma_p_identity}, we have
\begin{align*}
(p-1)^n S(\gamma) &= (p-1)^n (-p)^{|\gamma A^{-1}|} \sum_{k \in \N^n} c_{(p-1)\gamma A^{-1}+(p-1)k}(-p)^{|k|} \\
                  &= (p-1)^n \prod_{i=1}^n \sum_{k_i \geq 0} c_{(p-1)\left( \gamma A^{-1} \right)_i + (p-1)k_i} (-p)^{k_i} \pi^{(p-1)  \left(\gamma A^{-1} \right)_i} \\
                  &= \prod_{i=1}^n p (-p)^{ \left( \gamma A^{-1} \right)_i - 1} \Gamma_p\left(  \left( \gamma A^{-1} \right)_i \right) \\
                  &= p^n (-p)^{\left| \gamma A^{-1} \right| - n} \Gamma_p \left( \gamma A^{-1} \right) \,,
\end{align*}
where in the second step we used $c_k = \prod_{i=1}^n c_{k_i}$ and $(-p)^{|\gamma A^{-1}|} = \prod_{i=1}^n \pi^{(p-1)(\gamma A^{-1})_i}$. This simplifies to the result.
\end{proof}

\begin{proposition}\label{claim:fermat-hypersurface}
Let $A=n I_n$ be a Fermat-type Berglund--H\"ubsch matrix and suppose that $n \mid (p-1)$. Then
\[
{N}_{MW}(X_A) = \mathrm{ST}_p(A,\langle J\rangle).
\]
\end{proposition}

\begin{proof}
The elements $(\gamma,\lambda)$ with $\delta(\gamma,\lambda)\ne 0$ and $\lambda\ne 0$ are of the form $(0,kJ)$ for $k\in\{1,\ldots,n-1\}$. Each contributes $p^{{\rm age}(kJ)-1}$ to ${\rm ST}_p(A,\langle J\rangle)$, so their total contribution is  $p^0+p^1+\cdots+p^{n-2}=(p^{n-1}-1)/(p-1)$, matching the first term in \eqref{eq:mw-point-count}.

We will show that the second term of \eqref{eq:mw-point-count},
\begin{equation}\label{eq:second-term-formula}
\frac{(-1)^n}{p}
\sum_{I \subseteq \{1 , \ldots, n \}}
(-p)^{|\overline{I}|}
\sum_{\gamma} c_{(p-1)\gamma A^{-1}}(-p)^{|\gamma A^{-1}|},
\end{equation}
can be written as
\begin{equation}\label{eq:basis-cancel-term-decomposition}
\sum_{\substack{\gamma \in G^T \\ \delta(\gamma,0)=1}} (-1)^n p^{-1} (-p)^{|\gamma A^{-1}|}\Gamma_p(\gamma A^{-1}) + Z,
\end{equation}
where $Z$ is all remaining terms from \eqref{eq:second-term-formula}. In particular, we will show that $Z = 0$.

Elements $(\gamma,0) \in G^T \times G$ with $\delta(\gamma,0)\ne 0$ have $\gamma_i\in\{1,\ldots,n-1\}$ and satisfy $0<(\gamma A^{-1})_i<1$ for all $i$ by \cite[Corollary~6.3]{AP}. These $\gamma$ appear in every $I$-term of the sum \eqref{eq:second-term-formula}. Collecting $S(\gamma)$ from all $2^n$ subsets $I$ gives
\[
\frac{(-1)^n}{p}\,S(\gamma)\sum_{t=0}^n\binom{n}{t}(-p)^t = \frac{(p-1)^n\,S(\gamma)}{p}\,.
\]
Applying Lemma~\ref{lemma:s-gamma-gauss} yields $(-1)^n\,p^{-1}\,(-p)^{|\gamma A^{-1}|}\,\Gamma_p(\gamma A^{-1}) = (-1)^n \alpha(\gamma,0)$ since ${\dim(0)}=n$ and $|\gamma A^{-1}| = {\rm age}^\vee(\gamma)$.

So far, we've accounted for all of $\mathrm{ST}_p(A, \langle J \rangle)$. The remaining terms from \eqref{eq:mw-point-count} are
\begin{equation}\label{eq:z-sum}
Z = \frac{(-1)^n}{p}\sum_\alpha\sum_I (-p)^{|\overline{I}|}\,S(\gamma_{\alpha,I})\,,
\end{equation}
where we partitioned the summands, which come from $\gamma$ with $(\gamma A^{-1})_i\in\{0,1\}$ for some~$i$, into classes~$\alpha$: two elements $\gamma,\gamma'$ belong to the same class when $(\gamma A^{-1})_i = (\gamma' A^{-1})_i$ or both $(\gamma A^{-1})_i,(\gamma' A^{-1})_i\in\{0,1\}$, for all~$i$. Within each class, elements differ by $\sum t_i e_i A$ with $t_i\in\{0,1\}$.

We show $Z=0$ by cancelling terms in pairs within each class. The term $(-p)^{|\overline{I}|} S(\gamma_{\alpha,I})$ cancels with $(-p)^{|\overline{I \cup \{j\}}|} S(\gamma_{\alpha,I \cup \{j\}})$. In particular,
\[
\frac{(p-1)^n}{(-1)^n}
(-p)^{|\overline{I}|} S(\gamma_{\alpha,I})
= (-p)^{|\overline{I}|} (-p)^{|\gamma_{\alpha,I} A^{-1}|}\Gamma_p(\gamma_{\alpha,I} A^{-1})
\]
while, using $\gamma_{\alpha, I \cup \{j\}} = \gamma_{\alpha, I} + e_j A$, we have
\begin{align*}
\frac{(p-1)^n}{(-1)^n} (-p)^{|\overline{I \cup \{j\}}|} S(\gamma_{\alpha, I \cup \{j\}})
&=(-p)^{|\overline{I}|-1}(-p)^{|\gamma_{\alpha,I}A^{-1}|+1} \Gamma_p((\gamma_{\alpha,I} + e_jA)A^{-1}) \\
&=-(-p)^{|\overline{I}|} (-p)^{|\gamma_{\alpha,I} A^{-1}|} \Gamma_p(\gamma_{\alpha,I}A^{-1})\,,
\end{align*}
where the last line follows from the fact that
\begin{equation}\label{eq:gamma-pair-cancellation}
\Gamma_p((\gamma_{\alpha,I} + e_j A)A^{-1}) = \Gamma_p(\gamma_{\alpha,I}A^{-1})\Gamma_p(1)
\end{equation}
and $\Gamma_p(1) = -1$.

The proof is complete if we can show that all terms in $Z$ are accounted for in pairs that cancel in this fashion. For each class $\alpha$, choose a fixed index $j$ such that $(\gamma A^{-1})_j \in \{0, 1\}$ for all $\gamma$ in the class. We can partition the subsets $I \subseteq \{1,\ldots,n\}$ into non-overlapping pairs $(I, I \cup \{j\})$ where $j \notin I$. For each pair, the term corresponding to $\gamma_{\alpha, I}$ (since $j \notin I$, we have $(\gamma_{\alpha, I} A^{-1})_j = 0$) cancels with the term corresponding to $\gamma_{\alpha, I \cup \{j\}} = \gamma_{\alpha, I} + e_j A$, as shown above. Summing over all such pairs guarantees that the total contribution of the class is $0$, and thus $Z = 0$.
\end{proof}

\begin{remark}
The proof above is stated for $\F_p$, but it carries over to $\F_{p^\nu}$ for any $\nu\ge 1$ by the generalization described in Remark~\ref{rem:section4-generalization}.
\end{remark}

\begin{example}
Let $W_A(x) = x_1^4 + x_2^4 + x_3^4 + x_4^4$ and $G = \langle J \rangle$.
Then vertical orbifold cohomology comes from $\lambda = J, 2 J, 3 J$ which correspond to $h^{0,0}=1$, $h^{1,1}=1$, $h^{2,2} = 1$, respectively. The associated orbifold trace contributions are $1 + p + p^2$, which correspond to the first term of \eqref{eq:mw-point-count}. There are 21 elements of the form $(\gamma, 0)$ in $G^T \times G$. They satisfy $(\gamma A^{-1})_i \in \{ \frac{1}{4}, \frac{2}{4}, \frac{3}{4} \}$ and $|\gamma A^{-1}| \in \Z$ and account for the remaining terms from \eqref{eq:supertrace-nu}.
The possible values of $\gamma A^{-1}$ are listed up to permuation in Table~\ref{table:fermat-quartic-partitions}.

The remaining temrs in \eqref{eq:mw-point-count} include $S(\gamma)$ where some entry of $\gamma A^{-1}$ is in $\{0, 1\}$.

Table~\ref{table:fermat-quartic-cancellation-partition-1} shows cancelling pairs in the partition corresponding to $(\frac{1}{4}, \frac{3}{4}, *, *)$.
Table~\ref{table:fermat-quartic-cancellation-partition-2} shows cancelling pairs in the partition corresponding to $(\frac{1}{2}, \frac{3}{4}, \frac{3}{4}, *)$.
Permutations of coordinates account for all terms in \eqref{eq:mw-point-count}.

 \renewcommand{\arraystretch}{1.5}
\begin{table}[ht]
\centering
\begin{tabular}{@{}cc@{}}
\toprule
$(*, *, *, *)$ &
$(*, *, \frac{1}{4}, \frac{3}{4})$ \\
$(*, *, \frac{1}{2}, \frac{1}{2})$ &
$(*, \frac{1}{4}, \frac{1}{4}, \frac{1}{2})$ \\
$(*, \frac{1}{2}, \frac{3}{4}, \frac{3}{4})$ &
$(\frac{1}{4}, \frac{1}{4}, \frac{1}{4}, \frac{1}{4})$ \\
$(\frac{1}{4}, \frac{1}{4}, \frac{3}{4}, \frac{3}{4})$ &
$(\frac{1}{4}, \frac{1}{2}, \frac{1}{2}, \frac{3}{4})$ \\
$(\frac{1}{2}, \frac{1}{2}, \frac{1}{2}, \frac{1}{2})$ &
$(\frac{3}{4}, \frac{3}{4}, \frac{3}{4}, \frac{3}{4})$ \\
\bottomrule
\end{tabular}
\caption{The values of $\gamma A^{-1}$ up to permuatation for $W_A = x_1^4 + x_2^4 + x_3^4 + x_4^4$. The symbol $*$ denotes either $0$ or $1$.}
\label{table:fermat-quartic-partitions}
\end{table}
 \renewcommand{\arraystretch}{1}

\begin{table}[ht]\scriptsize
\centering
\begin{tabular}{cccc@{\quad}|@{\quad}cccc}
\toprule
$I$ & $| \overline{I} |$ & $| \gamma A^{-1} |$ & $\gamma$ term &
$I \cup \{ j \}$ & $| \overline{I \cup \{ j\} } |$ & $| \gamma' A^{-1} |$ & $\gamma'$ term \\
\midrule
$\emptyset$ & $4$ & $1$ & $(-p)^{4} (-p)^{1} \Gamma_p(\frac{1}{4}, \frac{3}{4}, 0, 0)$ &
$\{3\}$ & $3$ & $2$ & $(-p)^{3} (-p)^{2} \Gamma_p(\frac{1}{4}, \frac{3}{4}, 1, 0)$ \\
$\{1\}$ & $3$ & $1$ & $(-p)^{3} (-p)^{1} \Gamma_p(\frac{1}{4}, \frac{3}{4}, 0, 0)$ &
$\{1, 3\}$ & $2$ & $2$ & $(-p)^{2} (-p)^{2} \Gamma_p(\frac{1}{4}, \frac{3}{4}, 1, 0)$ \\
$\{2\}$ & $3$ & $1$ & $(-p)^{3} (-p)^{1} \Gamma_p(\frac{1}{4}, \frac{3}{4}, 0, 0)$ &
$\{2, 3\}$ & $2$ & $2$ & $(-p)^{2} (-p)^{2} \Gamma_p(\frac{1}{4}, \frac{3}{4}, 1, 0)$ \\
$\{4\}$ & $3$ & $2$ & $(-p)^{3} (-p)^{2} \Gamma_p(\frac{1}{4}, \frac{3}{4}, 0, 1)$ &
$\{3, 4\}$ & $2$ & $3$ & $(-p)^{2} (-p)^{3} \Gamma_p(\frac{1}{4}, \frac{3}{4}, 1, 1)$ \\
$\{1, 2\}$ & $2$ & $1$ & $(-p)^{2} (-p)^{1} \Gamma_p(\frac{1}{4}, \frac{3}{4}, 0, 0)$ &
$\{1, 2, 3\}$ & $1$ & $2$ & $(-p)^{1} (-p)^{2} \Gamma_p(\frac{1}{4}, \frac{3}{4}, 1, 0)$ \\
$\{1, 4\}$ & $2$ & $2$ & $(-p)^{2} (-p)^{2} \Gamma_p(\frac{1}{4}, \frac{3}{4}, 0, 1)$ &
$\{1, 3, 4\}$ & $1$ & $3$ & $(-p)^{1} (-p)^{3} \Gamma_p(\frac{1}{4}, \frac{3}{4}, 1, 1)$ \\
$\{2, 4\}$ & $2$ & $2$ & $(-p)^{2} (-p)^{2} \Gamma_p(\frac{1}{4}, \frac{3}{4}, 0, 1)$ &
$\{2, 3, 4\}$ & $1$ & $3$ & $(-p)^{1} (-p)^{3} \Gamma_p(\frac{1}{4}, \frac{3}{4}, 1, 1)$ \\
$\{1, 2, 4\}$ & $1$ & $2$ & $(-p)^{1} (-p)^{2} \Gamma_p(\frac{1}{4}, \frac{3}{4}, 0, 1)$ &
$\{1, 2, 3, 4\}$ & $0$ & $3$ & $(-p)^{0} (-p)^{3} \Gamma_p(\frac{1}{4}, \frac{3}{4}, 1, 1)$ \\
\bottomrule
\end{tabular}
\caption{Cancellation pairs in the $(\frac{1}{4}, \frac{3}{4}, *, *)$ partition for $W_A = x_1^4 + x_2^4 + x_3^4 + x_4^4$.}
\label{table:fermat-quartic-cancellation-partition-1}
\end{table}

\begin{table}[ht]\scriptsize
\centering
\begin{tabular}{cccc@{\quad}|@{\quad}cccc}
\toprule
$I$ & $| \overline{I} |$ & $| \gamma A^{-1} |$ & $\gamma$ term &
$I \cup \{ j \}$ & $| \overline{I \cup \{ j\} } |$ & $| \gamma' A^{-1} |$ & $\gamma'$ term \\
\midrule
$\emptyset$ & $4$ & $2$ & $(-p)^{4} (-p)^{2} \Gamma_p(\frac{1}{2}, \frac{3}{4}, \frac{3}{4}, 0)$ &
$\{4\}$ & $3$ & $3$ & $(-p)^{3} (-p)^{3} \Gamma_p(\frac{1}{2}, \frac{3}{4}, \frac{3}{4}, 1)$ \\
$\{1\}$ & $3$ & $2$ & $(-p)^{3} (-p)^{2} \Gamma_p(\frac{1}{2}, \frac{3}{4}, \frac{3}{4}, 0)$ &
$\{1, 4\}$ & $2$ & $3$ & $(-p)^{2} (-p)^{3} \Gamma_p(\frac{1}{2}, \frac{3}{4}, \frac{3}{4}, 1)$ \\
$\{2\}$ & $3$ & $2$ & $(-p)^{3} (-p)^{2} \Gamma_p(\frac{1}{2}, \frac{3}{4}, \frac{3}{4}, 0)$ &
$\{2, 4\}$ & $2$ & $3$ & $(-p)^{2} (-p)^{3} \Gamma_p(\frac{1}{2}, \frac{3}{4}, \frac{3}{4}, 1)$ \\
$\{3\}$ & $3$ & $2$ & $(-p)^{3} (-p)^{2} \Gamma_p(\frac{1}{2}, \frac{3}{4}, \frac{3}{4}, 0)$ &
$\{3, 4\}$ & $2$ & $3$ & $(-p)^{2} (-p)^{3} \Gamma_p(\frac{1}{2}, \frac{3}{4}, \frac{3}{4}, 1)$ \\
$\{1, 2\}$ & $2$ & $2$ & $(-p)^{2} (-p)^{2} \Gamma_p(\frac{1}{2}, \frac{3}{4}, \frac{3}{4}, 0)$ &
$\{1, 2, 4\}$ & $1$ & $3$ & $(-p)^{1} (-p)^{3} \Gamma_p(\frac{1}{2}, \frac{3}{4}, \frac{3}{4}, 1)$ \\
$\{1, 3\}$ & $2$ & $2$ & $(-p)^{2} (-p)^{2} \Gamma_p(\frac{1}{2}, \frac{3}{4}, \frac{3}{4}, 0)$ &
$\{1, 3, 4\}$ & $1$ & $3$ & $(-p)^{1} (-p)^{3} \Gamma_p(\frac{1}{2}, \frac{3}{4}, \frac{3}{4}, 1)$ \\
$\{2, 3\}$ & $2$ & $2$ & $(-p)^{2} (-p)^{2} \Gamma_p(\frac{1}{2}, \frac{3}{4}, \frac{3}{4}, 0)$ &
$\{2, 3, 4\}$ & $1$ & $3$ & $(-p)^{1} (-p)^{3} \Gamma_p(\frac{1}{2}, \frac{3}{4}, \frac{3}{4}, 1)$ \\
$\{1, 2, 3\}$ & $1$ & $2$ & $(-p)^{1} (-p)^{2} \Gamma_p(\frac{1}{2}, \frac{3}{4}, \frac{3}{4}, 0)$ &
$\{1, 2, 3, 4\}$ & $0$ & $3$ & $(-p)^{0} (-p)^{3} \Gamma_p(\frac{1}{2}, \frac{3}{4}, \frac{3}{4}, 1)$ \\
\bottomrule
\end{tabular}
\caption{Cancellation pairs in the $(\frac{1}{2}, \frac{3}{4}, \frac{3}{4}, *)$ partition ($W_A = x_1^4 + x_2^4 + x_3^4 + x_4^4$).}
\label{table:fermat-quartic-cancellation-partition-2}
\end{table}

\end{example}

The proof of Proposition~\ref{claim:fermat-hypersurface} does not work verbatim for non-Fermat matrices $A$. For instance, for chain matrices $A$ we can have $\gamma_i > 0$ but $(\gamma A^{-1})_i = 0$ for some $i$, which cannot happen in the diagonal case. However, the general approach of grouping terms in \eqref{eq:z-sum} to cancel in pairs can be used on a case-by-case basis, as the $3$-chain example below illustrates.

\begin{example}
Let $W_A(x) = x_1^2x_2 + x_2^2x_3 + x_3^3$. Terms from \eqref{eq:mw-point-count} that correspond to the four basis elements of $H(\cB_A^G,D_A)$ arise from the Monsky--Washnitzer trace just as in the Fermat case. The vertical term contributions add up to $1 + p$, and the remaining two elements are associated to terms containing $S(1,1,1)$ and $S(1, 2, 3)$.
The remaining terms are again given by an equation similar to \eqref{eq:z-sum}, but with a different partitioning of summands of \eqref{eq:trace-diagonal} into $S$-sum terms.

We will form two groups of partitions in this case. In the first group are terms indexed by $(\gamma, I)$ such that $\gamma_i > 0$ if and only if $(\gamma A^{-1})_i > 0$ for all $i \in I$.
These cancel using \eqref{eq:gamma-pair-cancellation} in the same way as in the Fermat case, so we call them \textit{Fermat-like}.
Pairs of Fermat-like terms of the form $\gamma = \sum_{j} e_jA$ for some subset of indices $j$ are shown in pairs that cancel out in Table~\ref{table:3-chain-vertex-cancellation} and Figure~\ref{fig:3-chain-vertex-cancellation}.
We also have Fermat-like terms away from the vertices of the cone, but we can again use \eqref{eq:gamma-pair-cancellation}. These are summarized in Table~\ref{table:3-chain-0-1-2-cancellation}.

The remaining terms from the Monsky--Washnitzer trace can be indexed by $(\gamma, I)$ such that $\gamma_i > 0$ but $(\gamma A^{-1})_i = 0$ for some $i \in I$.
In this case, if we proceeded to form $S$-sums as we did previously, we would include some terms from the Monsky--Washnitzer trace multiple times.
We consequently restrict the summation indices in order to account for each term from \eqref{eq:mw-point-count} exactly once.
For a set of indices $\delta \subset\{ 1, 2, 3\}$ and a partition representative $\gamma_{\alpha, I}$, define the sum
\begin{equation}\label{eq:partial-s-sum}
S^\delta(\gamma_{\alpha, I}) = \prod_{j \notin\delta} c_{(p-1)(\gamma_{\alpha, I} A^{-1})_j}(-p)^{(\gamma_{\alpha, I} A^{-1})_j} \prod_{j\in\delta}  \sum_{k_j \geq 0} c_{(p-1)\left[ (\gamma_{\alpha, I} A^{-1})_j + k_j \right]}  (-p)^{(\gamma_{\alpha, I} A^{-1})_j + k_j}.
\end{equation}
An enumeration of the unaccounted-for sums from the trace shows that the remaining terms can be written as
\begin{equation}\label{eq:chain-delta-z-sums}
\sum_{\alpha} \sum_{I} (-p)^{|\overline{I}|} S^{\delta} (\gamma_{\alpha, I}),
\end{equation}
for the specific $\delta$ and $\gamma_{\alpha,I}$ listed in Table~\ref{table:3-chain-non-fermat-cancellation}. Moreover, the table also shows which terms from \eqref{eq:chain-delta-z-sums} cancel in pairs, which can be straightforwardly seen from \eqref{eq:partial-s-sum}.

\begin{figure}[ht]
\centering

\tikzset{every picture/.style={line width=0.75pt}} 

\begin{tikzpicture}[x=0.75pt,y=0.75pt,yscale=-1,xscale=1]

\draw [color={rgb, 255:red, 155; green, 155; blue, 155 }  ,draw opacity=1 ]   (200,43) -- (200,220) ;
\draw [shift={(200,40)}, rotate = 90] [fill={rgb, 255:red, 155; green, 155; blue, 155 }  ,fill opacity=1 ][line width=0.08]  [draw opacity=0] (5.36,-2.57) -- (0,0) -- (5.36,2.57) -- cycle    ;
\draw [color={rgb, 255:red, 155; green, 155; blue, 155 }  ,draw opacity=1 ]   (377,220) -- (200,220) ;
\draw [shift={(380,220)}, rotate = 180] [fill={rgb, 255:red, 155; green, 155; blue, 155 }  ,fill opacity=1 ][line width=0.08]  [draw opacity=0] (5.36,-2.57) -- (0,0) -- (5.36,2.57) -- cycle    ;
\draw [color={rgb, 255:red, 155; green, 155; blue, 155 }  ,draw opacity=1 ]   (200,160) -- (200,220) ;
\draw [shift={(200,190)}, rotate = 270] [color={rgb, 255:red, 155; green, 155; blue, 155 }  ,draw opacity=1 ][line width=0.75]    (0,3.35) -- (0,-3.35)   ;
\draw [shift={(200,160)}, rotate = 270] [color={rgb, 255:red, 155; green, 155; blue, 155 }  ,draw opacity=1 ][line width=0.75]    (0,3.35) -- (0,-3.35)   ;
\draw [color={rgb, 255:red, 155; green, 155; blue, 155 }  ,draw opacity=1 ]   (200,70) -- (200,130) ;
\draw [shift={(200,130)}, rotate = 270] [color={rgb, 255:red, 155; green, 155; blue, 155 }  ,draw opacity=1 ][line width=0.75]    (0,3.35) -- (0,-3.35)   ;
\draw [shift={(200,100)}, rotate = 270] [color={rgb, 255:red, 155; green, 155; blue, 155 }  ,draw opacity=1 ][line width=0.75]    (0,3.35) -- (0,-3.35)   ;
\draw [shift={(200,70)}, rotate = 270] [color={rgb, 255:red, 155; green, 155; blue, 155 }  ,draw opacity=1 ][line width=0.75]    (0,3.35) -- (0,-3.35)   ;
\draw [color={rgb, 255:red, 155; green, 155; blue, 155 }  ,draw opacity=1 ]   (290,220) -- (230,220) ;
\draw [shift={(230,220)}, rotate = 360] [color={rgb, 255:red, 155; green, 155; blue, 155 }  ,draw opacity=1 ][line width=0.75]    (0,3.35) -- (0,-3.35)   ;
\draw [shift={(260,220)}, rotate = 360] [color={rgb, 255:red, 155; green, 155; blue, 155 }  ,draw opacity=1 ][line width=0.75]    (0,3.35) -- (0,-3.35)   ;
\draw [shift={(290,220)}, rotate = 360] [color={rgb, 255:red, 155; green, 155; blue, 155 }  ,draw opacity=1 ][line width=0.75]    (0,3.35) -- (0,-3.35)   ;
\draw [color={rgb, 255:red, 155; green, 155; blue, 155 }  ,draw opacity=1 ]   (380,220) -- (320,220) ;
\draw [shift={(320,220)}, rotate = 360] [color={rgb, 255:red, 155; green, 155; blue, 155 }  ,draw opacity=1 ][line width=0.75]    (0,3.35) -- (0,-3.35)   ;
\draw [shift={(350,220)}, rotate = 360] [color={rgb, 255:red, 155; green, 155; blue, 155 }  ,draw opacity=1 ][line width=0.75]    (0,3.35) -- (0,-3.35)   ;
\draw [color={rgb, 255:red, 155; green, 155; blue, 155 }  ,draw opacity=1 ]   (347.5,121.66) -- (200,220) ;
\draw [shift={(350,120)}, rotate = 146.31] [fill={rgb, 255:red, 155; green, 155; blue, 155 }  ,fill opacity=1 ][line width=0.08]  [draw opacity=0] (5.36,-2.57) -- (0,0) -- (5.36,2.57) -- cycle    ;
\draw [color={rgb, 255:red, 155; green, 155; blue, 155 }  ,draw opacity=1 ]   (225,201) -- (225,207) ;
\draw [color={rgb, 255:red, 155; green, 155; blue, 155 }  ,draw opacity=1 ]   (250,184) -- (250,190) ;
\draw [color={rgb, 255:red, 155; green, 155; blue, 155 }  ,draw opacity=1 ]   (275,167) -- (275,173) ;
\draw [color={rgb, 255:red, 155; green, 155; blue, 155 }  ,draw opacity=1 ]   (300,151) -- (300,157) ;
\draw [color={rgb, 255:red, 155; green, 155; blue, 155 }  ,draw opacity=1 ]   (325,134) -- (325,140) ;
\draw [color={rgb, 255:red, 155; green, 155; blue, 155 }  ,draw opacity=1 ] [dash pattern={on 0.84pt off 2.51pt}]  (284,204) -- (260,220) ;
\draw [color={rgb, 255:red, 155; green, 155; blue, 155 }  ,draw opacity=1 ] [dash pattern={on 0.84pt off 2.51pt}]  (250,157) -- (250,190) ;
\draw [color={rgb, 255:red, 155; green, 155; blue, 155 }  ,draw opacity=1 ] [dash pattern={on 0.84pt off 2.51pt}]  (250,157) -- (200,190) ;
\draw [color={rgb, 255:red, 155; green, 155; blue, 155 }  ,draw opacity=1 ] [dash pattern={on 0.84pt off 2.51pt}]  (284,204) -- (226,204) ;
\draw    (284,204) -- (200,220) ;
\draw [shift={(242,212)}, rotate = 169.22] [fill={rgb, 255:red, 0; green, 0; blue, 0 }  ][line width=0.08]  [draw opacity=0] (5.36,-2.57) -- (0,0) -- (5.36,2.57) -- (3.56,0) -- cycle    ;
\draw    (200,130) -- (200,220) ;
\draw [shift={(200,220)}, rotate = 90] [color={rgb, 255:red, 0; green, 0; blue, 0 }  ][fill={rgb, 255:red, 0; green, 0; blue, 0 }  ][line width=0.75]      (0, 0) circle [x radius= 1.34, y radius= 1.34]   ;
\draw [shift={(200,130)}, rotate = 90] [color={rgb, 255:red, 0; green, 0; blue, 0 }  ][fill={rgb, 255:red, 0; green, 0; blue, 0 }  ][line width=0.75]      (0, 0) circle [x radius= 1.34, y radius= 1.34]   ;
\draw    (250,157) -- (200,220) ;
\draw    (250,67) -- (250,157) ;
\draw [shift={(250,157)}, rotate = 90] [color={rgb, 255:red, 0; green, 0; blue, 0 }  ][fill={rgb, 255:red, 0; green, 0; blue, 0 }  ][line width=0.75]      (0, 0) circle [x radius= 1.34, y radius= 1.34]   ;
\draw [shift={(250,67)}, rotate = 90] [color={rgb, 255:red, 0; green, 0; blue, 0 }  ][fill={rgb, 255:red, 0; green, 0; blue, 0 }  ][line width=0.75]      (0, 0) circle [x radius= 1.34, y radius= 1.34]   ;
\draw    (284,114) -- (284,204) ;
\draw [shift={(284,204)}, rotate = 90] [color={rgb, 255:red, 0; green, 0; blue, 0 }  ][fill={rgb, 255:red, 0; green, 0; blue, 0 }  ][line width=0.75]      (0, 0) circle [x radius= 1.34, y radius= 1.34]   ;
\draw [shift={(284,114)}, rotate = 90] [color={rgb, 255:red, 0; green, 0; blue, 0 }  ][fill={rgb, 255:red, 0; green, 0; blue, 0 }  ][line width=0.75]      (0, 0) circle [x radius= 1.34, y radius= 1.34]   ;
\draw    (284,114) -- (200,130) ;
\draw [shift={(242,122)}, rotate = 169.22] [fill={rgb, 255:red, 0; green, 0; blue, 0 }  ][line width=0.08]  [draw opacity=0] (5.36,-2.57) -- (0,0) -- (5.36,2.57) -- (3.56,0) -- cycle    ;
\draw    (334,51) -- (250,67) ;
\draw [shift={(292,59)}, rotate = 169.22] [fill={rgb, 255:red, 0; green, 0; blue, 0 }  ][line width=0.08]  [draw opacity=0] (5.36,-2.57) -- (0,0) -- (5.36,2.57) -- (3.56,0) -- cycle    ;
\draw    (334,51) -- (284,114) ;
\draw    (250,67) -- (200,130) ;
\draw [color={rgb, 255:red, 155; green, 155; blue, 155 }  ,draw opacity=1 ] [dash pattern={on 0.84pt off 2.51pt}]  (334,141) -- (334,174) ;
\draw    (334,51) -- (334,141) ;
\draw [shift={(334,141)}, rotate = 90] [color={rgb, 255:red, 0; green, 0; blue, 0 }  ][fill={rgb, 255:red, 0; green, 0; blue, 0 }  ][line width=0.75]      (0, 0) circle [x radius= 1.34, y radius= 1.34]   ;
\draw [shift={(334,51)}, rotate = 90] [color={rgb, 255:red, 0; green, 0; blue, 0 }  ][fill={rgb, 255:red, 0; green, 0; blue, 0 }  ][line width=0.75]      (0, 0) circle [x radius= 1.34, y radius= 1.34]   ;
\draw    (334,141) -- (284,204) ;
\draw    (334,141) -- (250,157) ;
\draw [shift={(292,149)}, rotate = 169.22] [fill={rgb, 255:red, 0; green, 0; blue, 0 }  ][line width=0.08]  [draw opacity=0] (5.36,-2.57) -- (0,0) -- (5.36,2.57) -- (3.56,0) -- cycle    ;

\draw (182,214) node [anchor=north west][inner sep=0.75pt]  [font=\scriptsize] [align=left] {(a)};
\draw (287,198) node [anchor=north west][inner sep=0.75pt]  [font=\scriptsize] [align=left] {(b)};
\draw (233,150) node [anchor=north west][inner sep=0.75pt]  [font=\scriptsize] [align=left] {(c)};
\draw (337,135) node [anchor=north west][inner sep=0.75pt]  [font=\scriptsize] [align=left] {(d)};
\draw (181,124) node [anchor=north west][inner sep=0.75pt]  [font=\scriptsize] [align=left] {(e)};
\draw (287,109) node [anchor=north west][inner sep=0.75pt]  [font=\scriptsize] [align=left] {(f)};
\draw (232,60) node [anchor=north west][inner sep=0.75pt]  [font=\scriptsize] [align=left] {(g)};
\draw (337,45) node [anchor=north west][inner sep=0.75pt]  [font=\scriptsize] [align=left] {(h)};
\draw (372,207.4) node [anchor=north west][inner sep=0.75pt]  [font=\scriptsize]  {$\textcolor[rgb]{0.5,0.5,0.5}{x_{1}}$};
\draw (342,110.4) node [anchor=north west][inner sep=0.75pt]  [font=\scriptsize]  {$\textcolor[rgb]{0.5,0.5,0.5}{x}\textcolor[rgb]{0.5,0.5,0.5}{_{2}}$};
\draw (184,33.4) node [anchor=north west][inner sep=0.75pt]  [font=\scriptsize]  {$\textcolor[rgb]{0.5,0.5,0.5}{x}\textcolor[rgb]{0.5,0.5,0.5}{_{3}}$};

\end{tikzpicture}
\caption{Fundamental region of the lattice $\Z^3 /\Z^3 A$ ($W_A = x_1^2 x_2 + x_2^2 x_3 + x_3^3$).} \label{fig:3-chain-vertex-cancellation}
\end{figure}
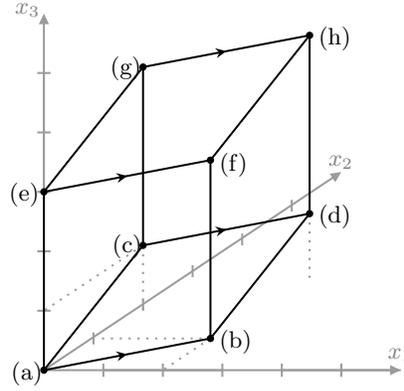

\begin{table}[ht]\footnotesize
\centering
\begin{tabular}{cccc@{\quad}|@{\quad}cccc}
\toprule
label & $I$ & $\gamma$ & $\gamma$ term &
label & $I \cup \{ j \}$ & $\gamma'$ & $\gamma'$ term \\
\midrule
(a) & $\emptyset$ & $(0,0,0)$ & $(-p)^{ 3 } \Gamma_p(0,0,0)$ & (b) & $\{1\}$ & $(2,1,0)$ & $(-p) (-p)^{ 2 } \Gamma_p(1,0,0)$ \\
(c) & $\{2\}$ & $(0,2,1)$ & $(-p) (-p)^{ 2 } \Gamma_p(0,1,0)$ & (d) & $\{1,2\}$ & $(2,3,1)$ & $(-p)^{ 2 } (-p) \Gamma_p(1,1,0)$ \\
(e) & $\{3\}$ & $(0,0,3)$ & $(-p) (-p)^{ 2 } \Gamma_p(0,0,1)$ & (f) & $\{1,3\}$ & $(2,1,3)$ & $(-p)^{ 2 } (-p) \Gamma_p(1,0,1)$ \\
(g) & $\{2,3\}$ & $(0,2,4)$ & $(-p)^{ 2 } (-p) \Gamma_p(0,1,1)$ & (h) & $\{1,2,3\}$ & $(2,3,4)$ & $(-p)^{ 3 }  \Gamma_p(1,1,1)$ \\
\bottomrule
\end{tabular}
\caption{Cancellation pairs corresponding to vertices of the fundamental region of the lattice $\Z^3 / \Z^3 A$ ($W_A = x_1^2 x_2 + x_2^2 x_3 + x_3^3$).}
\label{table:3-chain-vertex-cancellation}
\end{table}

\begin{table}[ht]\footnotesize
\centering
\begin{tabular}{ccc@{\quad}|@{\quad}ccc}
\toprule
$I$ & $\gamma$ & $\gamma$ term &
$I \cup \{ j \}$ & $\gamma'$ & $\gamma'$ term \\
\midrule
$\emptyset$ & $(0,1,2)$ & $(-p) (-p)^{ 3 } \Gamma_p(0,\frac{1}{2},\frac{1}{2})$ & $\{1\}$ & $(2,2,2)$ & $(-p)^{ 2 } (-p)^{ 2 } \Gamma_p(1,\frac{1}{2},\frac{1}{2})$ \\
$\{2\}$ & $(0,3,3)$ & $(-p)^{ 2 } (-p)^{ 2 } \Gamma_p(0,\frac{3}{2},\frac{1}{2})$ & $\{1,2\}$ & $(2,4,3)$ & $(-p)^{ 3 } (-p) \Gamma_p(1,\frac{3}{2},\frac{1}{2})$ \\
$\{3\}$ & $(0,1,5)$ & $(-p)^{ 2 } (-p)^{ 2 } \Gamma_p(0,\frac{1}{2},\frac{3}{2})$ & $\{1,3\}$ & $(2,2,5)$ & $(-p)^{ 3 } (-p) \Gamma_p(1,\frac{1}{2},\frac{3}{2})$ \\
$\{2,3\}$ & $(0,3,6)$ & $(-p)^{ 3 } (-p) \Gamma_p(0,\frac{3}{2},\frac{3}{2})$ & $\{1,2,3\}$ & $(2,4,6)$ & $(-p)^{ 4 }  \Gamma_p(1,\frac{3}{2},\frac{3}{2})$ \\
\bottomrule
\end{tabular}
\caption{Cancellation pairs for $W_A = x_1^2 x_2 + x_2^2 x_3 + x_3^3$ corresponding to $\gamma = (0,1,2) + \sum_i t_i e_i A$ in various $I$-regions.}
\label{table:3-chain-0-1-2-cancellation}
\end{table}

\begin{table}[ht]\footnotesize
\centering
\begin{tabular}{ccc@{\quad}|@{\quad}ccc}
\toprule
$I$ & $\gamma$ & $\gamma$ term &
$I \cup \{ j \}$ & $\gamma'$ & $\gamma'$ term \\
\midrule
$\{2\}$ & $(2,1,0)$ & $(-p)^{ 2 } S^{ \{ 1,3 \} }(1,0,0)$ & $\{2,3\}$ & $(2,1,3)$ & $(-p) S^{ \{ 1,3 \} }(1,0,1)$ \\
$\{1,2\}$ & $(2,1,0)$ & $(-p) S^{ \{ 1,3 \} }(1,0,0)$ & $\{1,2,3\}$ & $(2,1,3)$ & $S^{ \{ 1,3 \} }(1,0,1)$ \\
$\{3\}$ & $(0,2,1)$ & $(-p)^{ 2 } S^{ \{ 1,2 \} }(0,1,0)$ & $\{1,3\}$ & $(2,3,1)$ & $(-p) S^{ \{ 1,2 \} }(1,1,0)$ \\
$\{2,3\}$ & $(0,2,1)$ & $(-p) S^{ \{ 1,2 \} }(0,1,0)$ & $\{1,2,3\}$ & $(2,3,1)$ & $S^{ \{ 1,2 \} }(1,1,0)$ \\
\bottomrule
\end{tabular}
\caption{Cancellation pairs for $W_A = x_1^2 x_2 + x_2^2 x_3 + x_3^3$ corresponding to terms which don't appear in the Fermat case.}
\label{table:3-chain-non-fermat-cancellation}
\end{table}

\end{example}

\end{document}